\documentclass[11pt,oneside]{amsart}

\usepackage[usenames,dvipsnames]{xcolor}
\usepackage{amsmath,ifthen, amsfonts, amssymb,
srcltx, 
   amsopn,tikz, tkz-euclide, enumerate}
\usetikzlibrary{decorations.markings,arrows}

\usepackage{pdfsync}
\usepackage{url}
\usepackage{graphicx}

\newcommand{\showcomments}{yes}
\renewcommand{\showcomments}{no}

\newsavebox{\commentbox}
\newenvironment{com}%
{\ifthenelse{\equal{\showcomments}{yes}}%
{\footnotemark
        \begin{lrbox}{\commentbox}
        \begin{minipage}[t]{1.25in}\raggedright\sffamily\tiny
        \footnotemark[\arabic{footnote}]}
{\begin{lrbox}{\commentbox}}}%
{\ifthenelse{\equal{\showcomments}{yes}}%
{\end{minipage}\end{lrbox}\marginpar{\usebox{\commentbox}}}
{\end{lrbox}}}

\newtheorem{thm}{Theorem}[section]
\newtheorem{lem}[thm]{Lemma}

\newtheorem{cor}[thm]{Corollary}
\newtheorem{conj}[thm]{Conjecture}
\newtheorem{prop}[thm]{Proposition}

\theoremstyle{definition}
\newtheorem{defn}[thm]{Definition}
\newtheorem{rem}[thm]{Remark}
\newtheorem{exmp}[thm]{Example}

\newtheorem{prob}[thm]{Problem}
\newtheorem{guess}[thm]{Guess}

\DeclareMathOperator{\Out}{Out}

\DeclareMathOperator{\link}{link}

\newcommand{\homology}{\ensuremath{{\sf{H}}}}

\newcommand{\curvature}{{\ensuremath{\kappa}}}
\newcommand{\field}[1]{\mathbb{#1}}
\newcommand{\integers}{\ensuremath{\field{Z}}}

\newcommand{\reals}{\ensuremath{\field{R}}}

\newcommand{\hyperbolic}{\ensuremath{\field{H}}}

\newcommand{\alink}{\link_{\uparrow}}
\newcommand{\dlink}{\link_{\downarrow}}

\newcommand{\boundary}   {{\ensuremath \partial}}

\newcommand{\euler}{\chi}

\newcommand{\VF}{\mathcal{VF}}

\newcommand{\mcl}{\mathcal}

\DeclareMathOperator{\girth}{girth}
\DeclareMathOperator{\cd}{cd}
\DeclareMathOperator{\vcd}{vcd}

\newcommand{\bwstyle}{\tikzstyle{every node}=[circle, draw, fill,
                        inner sep=0pt]}
\tikzset{c/.style={every coordinate/.try}}

\begin{document}

\title{Virtually Fibering Right-Angled Coxeter Groups}

\author{Kasia Jankiewicz}
\author{Sergey Norin}
\author[D.~T.~Wise]{Daniel T. Wise}
	\address{Dept. of Math. \& Stats.\\
                    McGill University \\
                    Montreal, Quebec, Canada H3A 0B9}
           \email{kasia.jankiewicz@mail.mcgill.ca}
           \email{snorine@gmail.com}
           \email{wise@math.mcgill.ca}

\subjclass[2010]{20F55}
\keywords{Coxeter groups, Morse theory, Coherent groups}
\date{\today}
\thanks{Research supported by NSERC}

\begin{com}
{\bf \normalsize COMMENTS\\}
ARE\\
SHOWING!\\
\end{com}

\begin{abstract}
 We show that certain right-angled Coxeter groups have finite index subgroups that quotient to $\integers$
 with finitely generated kernels.
  The proof uses Bestvina-Brady Morse theory facilitated by combinatorial arguments.
 We describe a variety of examples where the plan succeeds or fails.
Among the successful examples are the right-angled reflection groups in $\hyperbolic^4$ with fundamental domain
the $120$-cell or the $24$-cell.
  \end{abstract}

\maketitle

\vspace{-1cm}
{\small
\tableofcontents
}

\section{Introduction}
A group $G$ \emph{virtually algebraically fibers} if there is a finite index subgroup $G'$ admitting an
epimorphism $G'\to\integers$ with finitely generated kernel. We do not require any other finiteness properties of
the kernel beyond finite generation. It is an interesting observation of Stallings \cite{Stallings62} that when $G$ is 	
the fundamental group of a $3$-manifold $M$ and $G$ virtually algebraically fibers then the kernel is the
fundamental group of a surface $S$, and the corresponding finite cover of $M$ is an $S$-bundle over a circle.
In fact, with the exception of a limited class of closed graph manifolds,
every compact irreducible 3-manifold $M$ with $\euler(M)=0$ does virtually fiber
\cite{WiseIsraelHierarchy, AgolGrovesManning2012, LiuGraphManifolds, PrzytyckiWiseMixed}.

The goal of this paper is  to obtain virtual algebraic fibering of a right-angled Coxeter group $G$.
The group $G$ acts properly and cocompactly on a CAT(0) cube complex $\widetilde X$.
The method of this paper is to use a combinatorial argument to choose a finite index  torsion-free subgroup $G'$, so that
there is a
function $X\rightarrow S^1$  on
the compact nonpositively curved cube complex $X= G' \backslash \widetilde X$ that lifts to a Bestvina-Brady Morse function $\widetilde X\to \reals.$
Although the situation can become quite complicated, our method enables us to produce Morse functions even when $G$ is associated to a simplicial graph $\Gamma$ that is quite large.

The initial motivation was to examine whether the Bestvina-Brady theory can be applied successfully for a hyperbolic 3-manifold in the simple setting of a right-angled hyperbolic reflection group.
In the uniform case, we find that it can be applied for an infinite family of dual L\"{o}bell graphs (see Lemma~\ref{lem:LobellGraphsWin}),
in the non-uniform case, we find that there are examples where Bestvina-Brady cannot be applied to \emph{any} finite index subgroup quotienting $\widetilde X$
(see Example~\ref{exmp:tutte}). This last realization is established by connecting the problem to an old conjecture of Tait about hamiltonian cycles in graphs,
to which a famous counterexample was provided by Tutte.
Overall our results provide some evidence that Bestvina-Brady Morse theory can be applied towards virtual fibering as an alternate route to Agol's criterion \cite{Agol08}.
However, Tutte's example demonstrates that alternate cube complexes would have to be utilized - even for certain cusped hyperbolic manifolds.

In two other highly noteworthy cases, we successfully apply the theory to a uniform and a non-uniform 4-dimensional hyperbolic reflection group:
namely, the reflection groups whose fundamental domain is the right-angled hyperbolic 120-cell, and the right-angled hyperbolic 24-cell.
These yield the first examples of higher-dimensional hyperbolic manifolds whose fundamental group virtually algebraically fiber.

Most of the paper is dedicated to examples which begin to substantiate the larger goal of obtaining virtual algebraic fibering of virtually special groups in a broader context than for 3-manifold groups.

{\bf Acknowledgement:} We are grateful to the referee for helpful comments and corrections.

\subsection{Contextualization}
Bestvina and Brady \cite{BestvinaBrady97} introduced their highly influential theory partially in response to a flawed attempt by Gromov to produce non-hyperbolic groups that do not contain $\integers^2$. Their theory generalizes earlier work of Stallings \cite{Stallings63} and Bieri \cite{BieriBook81} and led to many interesting examples of subgroups with exotic finiteness properties. Partially realizing Gromov's goal, Brady \cite{Brady99} applied this theory to obtain a remarkable example of a word-hyperbolic group with a finitely presented subgroup that is not word-hyperbolic. It remains an open problem whether there can be such a subgroup with a finite $K(\pi,1)$. We suggest a possible approach towards this in Section~\ref{sec:7}. Our approach creates a platform which enables the Bestvina-Brady theory to reach complicated examples. It seems that some exotic examples necessarily arise in a fairly involved setting.

Virtual algebraic fibering was investigated for Bourdon buildings in \cite{WiseRandomMorse} and for Coxeter groups having higher bounded exponents in \cite{JankiewiczWise2015}. In those cases the Bestvina-Brady Morse function was instead produced using the probabilistic method.

\subsection{Guide to reading this text:}

The ``combinatorial game'' that we are playing here is a self-contained problem about graphs
described in Section~\ref{sec:legal systems} and illustrated there by two simple examples. Its vocabulary, ``systems of moves'', ``legal states'' and so forth is used heavily throughout the text.

We review the right-angled Coxeter group $G(\Gamma)$ associated to the simplicial graph $\Gamma$ in Section~\ref{sub:racg}. We also describe there the CAT(0) cube complex $\widetilde X$ that $G$ acts on cocompactly, which is the Davis complex of $G$.
We describe a conjecture about algebraic virtual fibering and Euler characteristic for some of these
right-angled Coxeter groups.

A brief review of the part of Bestvina-Brady Morse theory that we utilize is described in Section~\ref{sub:bestvina brady}. In Section~\ref{sub:legal system fibers}, we present the focal point of the paper, Theorem~\ref{thm:win fiber}, which explains that a legal system for $\Gamma$ leads to virtual algebraic fibering of $G(\Gamma)$.

Section~\ref{sec:succeeds} exhibits a variety of examples where our method successfully provides algebraic virtual fibering.

Section~\ref{sec:fails} describes examples where our method fails, and indeed any attempt to virtually apply the Bestvina-Brady theory for these examples must fail.

Section~\ref{sec:7} poses a problem aiming to produce exotic subgroups of hyperbolic groups following Brady as discussed above.

In Section~\ref{sec:8} we investigate the applicability of the method to random right-angled Coxeter groups associated to  Erd\H{o}s-R\'enyi random graphs.

\section{Systems of moves and legal states}\label{sec:legal systems}

Let $\Gamma=\Gamma(V,E)$ be a simplicial graph with vertices $V$ and edges $E$.

A \emph{state} of $\Gamma$ is a subset $S\subset V$.
The state is \emph{legal} if the subgraphs induced by $S$ and by the complement $V-S$ of $S$ in $V$ are both nonempty and connected.
A \emph{move at} $v\in V$ is an element $m_v\in 2^V$ with the following property:
\begin{enumerate}
\item\label{move prop:1} $v\in m_v$.
\item\label{move prop:2} $u\not\in m_v$ if $\{u,v\}\in E$.
\end{enumerate}

A \emph{system} of moves is a choice $m_v$ of a move for each $v\in V$. We do not assume that $m_v\neq m_u$ for $v\neq u$.

We identify $\integers_2^V$ with $2^V$ in the obvious way where $\emptyset$ is the identity element
and multiplication is the symmetric difference.
%
Accordingly, each state and each move is identified with an element of $\integers_2^V$.
The set of moves generates a subgroup $M$ of $\integers_2^V$.
The system is \emph{legal} if there is an $M$-orbit all of whose elements are legal states.
We refer to such an orbit as a \emph{legal orbit}.

A state $S$ is \emph{strongly legal} if $S$ is legal, and moreover,
each vertex of $S$ is adjacent to a vertex of $V-S$ and vice-versa.
It follows from the definitions that every state in a legal orbit is actually strongly legal.

\begin{exmp}
In many cases, our system of moves arises as follows:  There is a partition $V=\sqcup_i V_i$ where each $V_i$ has the property that no two of its vertices are adjacent.
The move associated to $v$ is the element $V_i$ containing $v$. We refer to such a system of moves as a \emph{colored system}.
\end{exmp}

\begin{exmp}
Consider the following graph $\Gamma$ on four vertices and the system of moves given by the partition of $\Gamma$ into three parts:
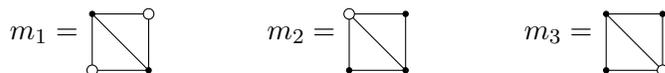
\begin{align*}
&&m_1 =
\bwstyle
\begin{tikzpicture}[baseline=2.5ex, scale=.75]
\draw[white, fill = white,minimum width=4pt]
{
	(0,0) node {}
	(1,1) node {}
};
\draw[black,minimum width=2pt]
{
	(0,1) node {} -- (1,1) node {} -- (1,0) node {} -- (0,0) node {} --(0,1) -- (1,0)
};
\draw[black, fill = white,minimum width=4pt]
{
	(0,0) node {}
	(1,1) node {}
};
\end{tikzpicture}
&&m_2 =
\bwstyle\begin{tikzpicture}[baseline=2.5ex, scale=.75]
\draw[white, fill = white,minimum width=4pt]
{
	(0,0) node {}
	(1,1) node {}
};
\draw[black,minimum width=2pt]
{
	(0,1) node {} -- (1,1) node {} -- (1,0) node {} -- (0,0) node {} --(0,1) -- (1,0)
};
\draw[black, fill = white,minimum width=4pt]
{
	(0,1) node {}
};
\end{tikzpicture}
&&m_3 =
\bwstyle\begin{tikzpicture}[baseline=2.5ex, scale=.75]
\draw[white, fill = white,minimum width=4pt]
{
	(0,0) node {}
	(1,1) node {}
};
\draw[black,minimum width=2pt]
{
	(0,1) node {} -- (1,1) node {} -- (1,0) node {} -- (0,0) node {} --(0,1) -- (1,0)
};
\draw[black, fill = white,minimum width=4pt]
{
	(1,0) node {}
};
\end{tikzpicture}&&
\end{align*}
This is a legal system since there is a legal orbit:
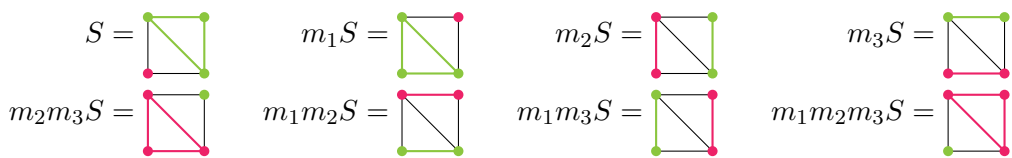
\begin{align*}
\bwstyle
S = \begin{tikzpicture}[baseline=2.5ex, scale=.75]
\draw[black]
{
	(1,0) to (0,0) to (0,1)
};
\draw[WildStrawberry, thick, fill = WildStrawberry, minimum width=3pt]
{
	(0,0) node {}
};
\draw[LimeGreen, thick, minimum width=3pt]
{
	(1,0) node {} -- (0,1) node {} -- (1,1) node {} -- (1,0)
};
\end{tikzpicture}&&
\bwstyle
m_1S = \begin{tikzpicture}[baseline=2.5ex, scale=.75]
\draw[black]
{
	(1,0) to (1,1) to (0,1)
};
\draw[WildStrawberry, thick, fill = WildStrawberry, minimum width=3pt]
{
	(1,1) node {}
};
\draw[LimeGreen, thick, minimum width=3pt]
{
	(1,0) node {} -- (0,1) node {} -- (0,0) node {} -- (1,0)
};
\end{tikzpicture}
&&\bwstyle
m_2S =  \begin{tikzpicture}[baseline=2.5ex, scale=.75]
\draw[black]
{
	(0,0) to (1,0) to (0,1) to (1,1)
};
\draw[WildStrawberry, thick, fill = WildStrawberry, minimum width=3pt]
{
	(0,0) node {} -- (0,1) node {}
};
\draw[LimeGreen, thick, minimum width=3pt]
{
	(1,1) node {} -- (1,0) node {}
};
\end{tikzpicture}
&&\bwstyle
m_3S =  \begin{tikzpicture}[baseline=2.5ex, scale=.75]
\draw[black]
{
	(0,0) to (0,1) to (1,0) to (1,1)
};
\draw[WildStrawberry, thick, fill = WildStrawberry, minimum width=3pt]
{
	(0,0) node {} -- (1,0) node {}
};
\draw[LimeGreen, thick, minimum width=3pt]
{
	(1,1) node {} -- (0,1) node {}
};
\end{tikzpicture}&
\\
\bwstyle
m_2m_3S =  \begin{tikzpicture}[baseline=2.5ex, scale=.75]
\draw[black]
{
	(1,0) to (1,1) to (0,1)
};
\draw[WildStrawberry, thick, minimum width=3pt]
{
	(0,0) node {} -- (1,0) node {} -- (0,1) node {} -- (0,0)
};
\draw[LimeGreen, thick, minimum width=3pt]
{
	(1,1) node {}
};
\end{tikzpicture}
&&\bwstyle
m_1m_2S =  \begin{tikzpicture}[baseline=2.5ex, scale=.75]
\draw[black]
{
	(0,0) to (0,1) to (1,0) to (1,1)
};
\draw[WildStrawberry, thick, minimum width=3pt]
{
	(0,1) node {} -- (1,1) node {}
};
\draw[LimeGreen, thick, minimum width=3pt]
{
	(1,0) node {} -- (0,0) node {}
};
\end{tikzpicture}
&&\bwstyle
m_1m_3S =  \begin{tikzpicture}[baseline=2.5ex, scale=.75]
\draw[black]
{
	(0,0) to (1,0) to (0,1) to (1,1)
};
\draw[LimeGreen, thick, minimum width=3pt]
{
	(0,0) node {} -- (0,1) node {}
};
\draw[WildStrawberry, thick, minimum width=3pt]
{
	(1,1) node {} -- (1,0) node {}
};
\end{tikzpicture}
&&\bwstyle
m_1m_2m_3S = \begin{tikzpicture}[baseline=2.5ex, scale=.75]
\draw[black]
{
	(1,0) to (0,0) to (0,1)
};
\draw[LimeGreen, thick, minimum width=3pt]
{
	(0,0) node {}
};
\draw[WildStrawberry, thick, minimum width=3pt]
{
	(1,0) node {} -- (0,1) node {} -- (1,1) node {} -- (1,0)
};
\end{tikzpicture}
\end{align*}
Note that the orbit of
$\bwstyle
S' =
\begin{tikzpicture}[baseline=2.5ex, scale=.75]
\draw[black]
{
	(1,1) to (1,0) to (0,0)
	(1,0) to (0,1)
};
\draw[WildStrawberry, thick, minimum width=3pt]
{
	(0,0) node {} -- (0,1) node {} -- (1,1) node {}
};
\draw[LimeGreen, thick, minimum width=3pt]
{
	(1,0) node {}
};
\end{tikzpicture}$
contains the non-legal state
$m_2 S' =
\tikzstyle{every node}=[circle, draw, fill,
                        inner sep=0pt, ]
\begin{tikzpicture}[baseline=2.5ex, scale=.75]
\draw[black]
{
	(1,1) to (1,0) to (0,0) to (0,1) to (1,1)
};
\draw[WildStrawberry, thick, minimum width=3pt]
{
	(0,0) node {}
	(1,1) node {}
};
\draw[LimeGreen, thick, minimum width=3pt]
{
	(1,0) node {} -- (0,1) node {}
};
\end{tikzpicture}$.
\end{exmp}

\section{Coxeter Groups, Curvature, and a Guess}
\newcommand{\NotTOCsubsection}[1]
{\subsection{#1}}
\NotTOCsubsection{Right-angled Coxeter Groups}\label{sub:racg}
Let $G(\Gamma)$ be the right-angled Coxeter group associated to $\Gamma$.
Let $\widetilde X$ denote the associated CAT(0) cube complex that $G$ acts on properly and cocompactly, which is known as the Davis complex of $G$.
We recall that the $1$-skeleton of $\widetilde X$ is isomorphic to the Cayley graph of $G$ after identifying each bigon to an edge, and $n$-cubes are equivariantly added to the $1$-skeleton for each collection of $n$ pairwise commuting generators (see for instance \cite[Chap. 7]{DavisCoxeterBook2008}).
\begin{com} Davis defines $X$ in section 1.2 of his book as the cubical complex $P_L$ associated to the simplicial complex $L$. Then he discussed $\tilde P_L$ and its CAT(0) structure.
\end{com}
Let $\alpha:G\rightarrow \integers_2^V$ denote the homomorphism induced by $s\mapsto\{s\}$,
so $\alpha$ is merely the abelianization homomorphism.
Let $G'=\ker(\alpha)$. Let $X = G'\backslash \widetilde X$.

\NotTOCsubsection{The Charney-Davis Curvature}
\begin{defn}[Curvature of $\Gamma$]
For a finite simplicial graph $\Gamma$ we define
its \emph{$n$-curvature} $\kappa_n(\Gamma) = \sum_{i=-1}^n (-2)^{i+1}|K_i|$ where $K_i$ is the set of $i$-cliques.
 Note that $|K_{-1}|=1$ as there is a unique $(-1)$-simplex.

We are specifically interested in $\kappa_2(\Gamma)=1-\frac{|V|}{2} +\frac{|E|}{4}$.
\end{defn}
Let $\kappa = \kappa_{\infty}$.
By distributing $\frac{1}{2^d}$ of the Euler characteristic $(-1)^d$ concentrated at an open $d$-cube among each of its $2^d$ vertices, we obtain the following tautological formula for a compact cube complex $Y$:
\begin{equation}\label{eq:euler curvature}
\euler(Y)=\sum_{y\in Y^0} \curvature(\link(y))
\end{equation}

For $X$ and $\Gamma$ as in Section~\ref{sub:racg} we have:
 $\euler(X)= 2^d\curvature(\Gamma)$.

\NotTOCsubsection{A Guess}
As mentioned in the introduction, a primary aspiration of this paper is to test the possibility of understanding virtual algebraic fibering of special groups
by applying Bestvina-Brady theory to finite covers of nonpositively curved cube complexes.
This was the original intended approach towards the virtual fibering problem by one of the authors. It was side-stepped by Agol who gave a criterion for virtual fibering that employed Gabai's sutured-manifold technology \cite{Agol08}.
Other forays that tested this approach using the probabilistic method were given in \cite{WiseRandomMorse,JankiewiczWise2015}.

We are guided by the following optimistic guess.
Although our approach confirms this guess in many cases, we show the approach is not always applicable
in Section~\ref{sub:duals}.
\begin{guess}\label{guess:naive}
Let $G(\Gamma)$ be a finitely generated right-angled Coxeter group.
Suppose
\begin{com} There might be a way of restating this in terms of properties of $\Gamma$.
\end{com}
$\kappa_2(\Gamma)\geq 0$. Then either:
\begin{enumerate}
\item  $G$ is virtually abelian;
\item $G$ splits over a virtually abelian subgroup;
\item $G$ has a non virtually abelian sub-Coxeter group $G'$ that virtually algebraically fibers.
\end{enumerate}
\end{guess}
We are grateful to Mike Davis for drawing attention to the connection with the conjecture in \cite{DavisOkun2001} that the orbifold associated to a right-angled Coxeter group $G(\Gamma)$ has a finite index cover fibering over a circle when the flag complex of $\Gamma$ is an even dimensional sphere.

\section{Legal Systems Provide Virtual Fiberings}

\subsection{Bestvina-Brady Morse Theory}\label{sub:bestvina brady}
A \emph{diagonal map} $[0,1]^d\rightarrow S^1$ on a $d$-cube is the restriction of the
composition $\reals^d \rightarrow \reals \rightarrow S^1$ where the first map
is the linear map  $(x_1,\dots, x_d)\mapsto \sum_i \pm x_i$ and the second map is the quotient $\reals / \integers
=S^1$.
 A \emph{diagonal map} on a cube complex is a map $X\rightarrow S^1$ whose restriction to each cube of $X$
 is a diagonal map.

Consider a diagonal map $\phi:X\to S^1$, and let $\tilde \phi:\widetilde X \to \reals$.
  A $(d-1)$-simplex of $\link(x)$ is \emph{ascending} (resp. \emph{descending}) if the restriction of $\tilde \phi$ to the corresponding $d$-cube has a minimum (resp. maximum) at $x$. The \emph{ascending link}  $\alink(x)$ (resp. \emph{descending link} $\dlink(x)$) is the subcomplex of $\link(x)$ consisting of all ascending (resp. descending) vertices and edges. Bestvina and Brady proved the following in \cite{BestvinaBrady97}:
\begin{thm}
If each ascending and descending link is connected then the kernel of $\phi_*:\pi_1X\rightarrow \integers$
is finitely generated.
\end{thm}

\begin{rem}
 A diagonal map is determined by (and equivalent to) directing the $1$-cubes of $X^1$ so that opposite $1$-cubes of each $2$-cube agree.
From this viewpoint, the ascending and descending links correspond to the induced subcomplexes
of outgoing and incoming 1-cubes.

Since connectedness of a simplicial complex is determined by the connectedness of its 1-skeleton we will
focus entirely on the 1-skeleton when discussing the ascending and descending links.
\end{rem}

\subsection{Virtually Algebraically Fibering $G(\Gamma)$}
\label{sub:legal system fibers}
\begin{thm}\label{thm:win fiber}
Let $\Gamma$ be a finite graph.
Suppose there is a system $m:V\rightarrow 2^V$ with a legal orbit.
Then there is a
diagonal 
map $\phi:X\rightarrow S^1$ whose ascending and descending links
are nonempty and connected.
\end{thm}
\begin{cor}\label{cor:win f.g.}
$G$ has a finite index subgroup $G'$ such that there is an epimorphism $G'\rightarrow \integers$
with finitely generated kernel.
\end{cor}

\begin{proof}[Proof of Theorem~\ref{thm:win fiber}]
Let $S$ be a state such that each element of $\langle m_v : v\in V\rangle S$ is legal.
Consider a base 0-cube $\hat x\in X^0$.
For each $z\in \integers_2^v$ we associate $z\hat x$  with $zS$.
This determines a way to direct the 1-cubes at $z\hat x$.
Specifically: A 1-cube is outgoing if the corresponding vertex $v$ of $\link(z\hat x)$ is in $zS$
and it is incoming if $v\not\in zS$.

Let $z_1\hat x$ and $z_2\hat x$ be the endpoints of a 1-cube $c$ corresponding to a move $v$ of $G$,
and note that $z_1= m_vz_2$.
Since $z_1=m_vz_2$ we see that the direction of $c$ induced by $z_1S$ and $z_2S$ agree by Move Property~\eqref{move prop:1}.
 Move Property~\eqref{move prop:2} implies that the opposite sides of each 2-cube are directed consistently.
 \end{proof}

\begin{rem}
We have restricted ourselves to Bestvina-Brady Morse functions associated to diagonal maps. These
correspond to those for which the minimum and maximum vertices of each square are antipodal.
In fact, a slight generalization of legal system is equivalent to
the existence of a diagonal Bestvina-Brady Morse function on a finite cover.
Diagonal maps are associated to a consistent direction for edges that are parallel along each ``hyperplane''
There are other possible Morse functions, but they are less natural to investigate in terms of $\Gamma$.
\end{rem}

\section{Examples where the method succeeds}\label{sec:succeeds}

\subsection{Some favorite small examples}\label{sub:small examples}
\subsubsection{Cube}\label{ex:cube}
Let $\Gamma$ be a $1$-skeleton of the $3$-cube. This is a bipartite graph, with $\girth(\Gamma) =4$ and $\kappa(\Gamma) = 1-\frac{8}{2}+\frac{12}{4}=0$. The system of moves corresponding to the bipartition is legal:
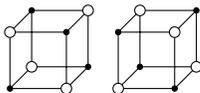
\[
\bwstyle\begin{tikzpicture}[scale=0.75]
\draw
{
	(0,0,0) to (0,1,0) to (1,1,0) to (1,0,0) to (0,0,0) to (0,0,1) to (1,0,1) to (1,1,1) to (0,1,1) to (0,0,1)
	(0,1,0) to (0,1,1)
	(1,1,0) to (1,1,1)
	(1,0,0) to (1,0,1)
};
\draw[black, fill = white,minimum width=4pt]
{
	(0,0,0) node {}
	(1,1,0) node {}
	(0,1,1) node {}
	(1,0,1) node {}
};
\draw[black,minimum width=2pt]
{
	(0,1,0) node {}
	(1,0,0) node {}
	(0,0,1) node {}
	(1,1,1) node {}
};
\end{tikzpicture}
\hspace{7pt}
\begin{tikzpicture}[scale=0.75]
\draw
{
	(0,0,0) to (0,1,0) to (1,1,0) to (1,0,0) to (0,0,0) to (0,0,1) to (1,0,1) to (1,1,1) to (0,1,1) to (0,0,1)
	(0,1,0) to (0,1,1)
	(1,1,0) to (1,1,1)
	(1,0,0) to (1,0,1)
};
\draw[black, fill = white,minimum width=4pt]
{
	(0,1,0) node {}
	(1,0,0) node {}
	(0,0,1) node {}
	(1,1,1) node {}
};
\draw[black,minimum width=2pt]
{
	(0,0,0) node {}
	(1,1,0) node {}
	(0,1,1) node {}
	(1,0,1) node {}
};
\end{tikzpicture}
\]
This is an example of a legal orbit:
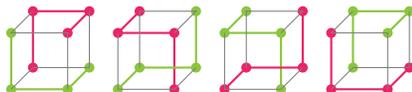
\[
\tikzstyle{every node}=[circle, draw, fill,
                        inner sep=0pt, minimum width=3pt]
\begin{tikzpicture}[scale = 0.75]
\draw[WildStrawberry,thick]
{
	(0,0,0) node {} -- (0,1,0) node {} -- (1,1,0) node {} -- (1,1,1) node {}
};
\draw[LimeGreen,thick]
{
	(1,0,0) node {} -- (1,0,1) node {} -- (0,0,1) node {} -- (0, 1,1) node {}
};
\draw[gray]
{
	(0,0,1) -- (0,0,0) -- (1,0,0) -- (1,1,0)
	(0,1,0) -- (0,1,1) -- (1,1,1) -- (1,0,1)
};
\end{tikzpicture}
\hspace{7pt}
\begin{tikzpicture}[scale = 0.75]
\draw[WildStrawberry,thick]
{
	(0,1,0) node {} -- (0,1,1) node {} -- (1,1,1) node {} -- (1,0,1) node {}
};
\draw[LimeGreen,thick]
{
	(0,0,1) node {} -- (0,0,0) node {} -- (1,0,0) node {} -- (1, 1,0) node {}
};
\draw[gray]
{
	(0,0,0) -- (0,1,0) -- (1,1,0) -- (1,1,1)
	(1,0,0) -- (1,0,1) -- (0,0,1) -- (0,1,1)
};
\end{tikzpicture}
\hspace{7pt}
\begin{tikzpicture}[scale = 0.75]
\draw[LimeGreen,thick]
{
	(0,1,0) node {} -- (0,1,1) node {} -- (1,1,1) node {} -- (1,0,1) node {}
};
\draw[WildStrawberry,thick]
{
	(0,0,1) node {} -- (0,0,0) node {} -- (1,0,0) node {} -- (1, 1) node {}
};
\draw[gray]
{
	(0,0,0) -- (0,1,0) -- (1,1,0) -- (1,1,1)
	(1,0,0) -- (1,0,1) -- (0,0,1) -- (0,1,1)
};
\end{tikzpicture}
\hspace{7pt}
\begin{tikzpicture}[scale = 0.75]
\draw[LimeGreen,thick]
{
	(0,0,0) node {} -- (0,1,0) node {} -- (1,1,0) node {} -- (1,1,1) node {}
};
\draw[WildStrawberry,thick]
{
	(1,0,0) node {} -- (1,0,1) node {} -- (0,0,1) node {} -- (0,1,1) node {}
};
\draw[gray]
{
	(0,0,1) -- (0,0,0) -- (1,0,0) -- (1,1,0)
	(0,1,0) -- (0,1,1) -- (1,1,1) -- (1,0,1)
};
\end{tikzpicture}
\]

\subsubsection{Wagner graph}
The Wagner graph $\Gamma$ is the following $3$-valent graph on $8$ vertices:
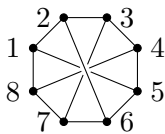
\[\tikzstyle{every node}=[circle, draw, fill,
                        inner sep=1pt,  ]
\begin{tikzpicture}[scale = 0.75]
\draw[black]
{
	(180-22.5:1) -- (135-22.5:1) -- (90-22.5:1) -- (45-22.5:1) -- (0-22.5:1) -- (315-22.5:1) -- (270-22.5:1) -- (225-22.5:1) -- (180-22.5:1)
	(180-22.5:1) -- (0-22.5:1)
	(90-22.5:1) -- (270-22.5:1)
};
\draw[white, fill = white, minimum width = 4pt, inner sep = 0pt]
{
	(0-22.5:0) node {}
};
\draw[black]
{
	(45-22.5:1) -- (225-22.5:1)
};
\draw[white, fill = white, minimum width = 2pt, inner sep = 0pt]
{
	(0-22.5:0) node {}
};
\draw[black]
{
	(135-22.5:1) -- (315-22.5:1)
};
\draw[black]
{
	(180-22.5:1) node [label=left:$1$] {}
	(135-22.5:1) node [label=left:$2$] {}
	(90-22.5:1) node [label=right:$3$] {}
	(45-22.5:1) node [label=right:$4$] {}
	(0-22.5:1) node [label=right:$5$] {}
	(315-22.5:1) node [label=right:$6$] {}
	(270-22.5:1) node [label=left:$7$] {}
	(225-22.5:1) node [label=left:$8$] {}
};
\end{tikzpicture}\]
The girth of the Wagner graph is $4$ and the curvature
$\kappa(\Gamma) = 1 - \frac{8}{2} + \frac{12}{4} = 0$.
The graph $\Gamma$ is noteworthy for having a legal system but not having any legal colored system.
The following system of moves is legal:
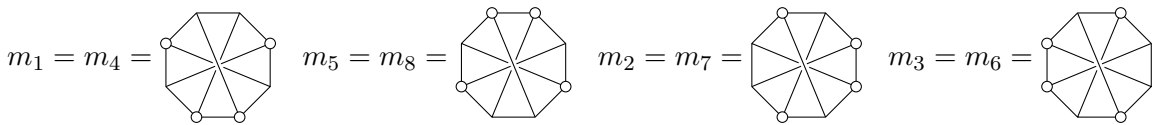
\begin{align*}
\bwstyle
m_1 =m_4 =
\begin{tikzpicture}[baseline = 0, scale = 0.75]
\draw[black]
{
	(180-22.5:1) -- (135-22.5:1) -- (90-22.5:1) -- (45-22.5:1) -- (0-22.5:1) -- (315-22.5:1) -- (270-22.5:1) -- (225-22.5:1) -- (180-22.5:1)
	(180-22.5:1) -- (0-22.5:1)
	(90-22.5:1) -- (270-22.5:1)
};
\draw[white, fill = white, minimum width = 4pt, inner sep = 0pt]
{
	(0-22.5:0) node {}
};
\draw[black]
{
	(45-22.5:1) -- (225-22.5:1)
};
\draw[white, fill = white, minimum width = 2pt, inner sep = 0pt]
{
	(0-22.5:0) node {}
};
\draw[black]
{
	(135-22.5:1) -- (315-22.5:1)
};
\draw[black, fill = white, minimum width = 4pt]
{
	(180-22.5:1) node {}
	(45-22.5:1) node {}
	(315-22.5:1) node {}
	(270-22.5:1) node {}
};
\end{tikzpicture}
&&
\bwstyle
m_5 = m_8 =
\begin{tikzpicture}[baseline=0, scale = 0.75]
\draw[black]
{
	(180-22.5:1) -- (135-22.5:1) -- (90-22.5:1) -- (45-22.5:1) -- (0-22.5:1) -- (315-22.5:1) -- (270-22.5:1) -- (225-22.5:1) -- (180-22.5:1)
	(180-22.5:1) -- (0-22.5:1)
	(90-22.5:1) -- (270-22.5:1)
};
\draw[white, fill = white, minimum width = 4pt, inner sep = 0pt]
{
	(0-22.5:0) node {}
};
\draw[black]
{
	(45-22.5:1) -- (225-22.5:1)
};
\draw[white, fill = white, minimum width = 2pt, inner sep = 0pt]
{
	(0-22.5:0) node {}
};
\draw[black]
{
	(135-22.5:1) -- (315-22.5:1)
};
\draw[black, fill = white, minimum width = 4pt]
{
	(135-22.5:1) node {}
	(90-22.5:1) node {}
	(0-22.5:1) node {}
	(225-22.5:1) node {}
};
\end{tikzpicture}
&&
\bwstyle
m_2 = m_7 =
\begin{tikzpicture}[baseline=0, scale = 0.75]
\draw[black]
{
	(180-22.5:1) -- (135-22.5:1) -- (90-22.5:1) -- (45-22.5:1) -- (0-22.5:1) -- (315-22.5:1) -- (270-22.5:1) -- (225-22.5:1) -- (180-22.5:1)
	(180-22.5:1) -- (0-22.5:1)
	(90-22.5:1) -- (270-22.5:1)
};
\draw[white, fill = white, minimum width = 4pt, inner sep = 0pt]
{
	(0-22.5:0) node {}
};
\draw[black]
{
	(45-22.5:1) -- (225-22.5:1)
};
\draw[white, fill = white, minimum width = 2pt, inner sep = 0pt]
{
	(0-22.5:0) node {}
};
\draw[black]
{
	(135-22.5:1) -- (315-22.5:1)
};
\draw[black, fill = white, minimum width = 4pt]
{
	(135-22.5:1) node {}
	(45-22.5:1) node {}
	(0-22.5:1) node {}
	(270-22.5:1) node {}
};
\end{tikzpicture}
&&
\bwstyle
m_3 = m_6 =
\begin{tikzpicture}[baseline=0, scale = 0.75]
\draw[black]
{
	(180-22.5:1) -- (135-22.5:1) -- (90-22.5:1) -- (45-22.5:1) -- (0-22.5:1) -- (315-22.5:1) -- (270-22.5:1) -- (225-22.5:1) -- (180-22.5:1)
	(180-22.5:1) -- (0-22.5:1)
	(90-22.5:1) -- (270-22.5:1)
};
\draw[white, fill = white, minimum width = 4pt, inner sep = 0pt]
{
	(0-22.5:0) node {}
};
\draw[black]
{
	(45-22.5:1) -- (225-22.5:1)
};
\draw[white, fill = white, minimum width = 2pt, inner sep = 0pt]
{
	(0-22.5:0) node {}
};
\draw[black]
{
	(135-22.5:1) -- (315-22.5:1)
};
\draw[black, fill = white, minimum width = 4pt]
{
	(90-22.5:1) node {}
	(180-22.5:1) node {}
	(225-22.5:1) node {}
	(315-22.5:1) node {}
};
\end{tikzpicture}
\end{align*}
The following figure indicates a legal orbit.
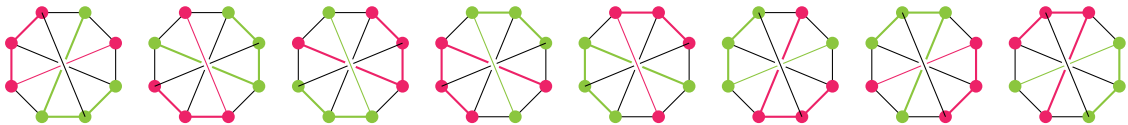
\begin{align*}
\bwstyle
\begin{tikzpicture}[baseline = 0, scale=0.75]
\draw[black]
{
	(135-22.5:1) -- (90-22.5:1) -- (45-22.5:1) -- (0-22.5:1)
	(270-22.5:1) -- (225-22.5:1)
	(180-22.5:1) -- (0-22.5:1)
};
\draw[WildStrawberry, thick, minimum width = 4pt,]
{
	(135-22.5:1) node {} -- (180-22.5:1) node {} -- (225-22.5:1) node {}
	(45-22.5:1) node {}
};
\draw[LimeGreen, thick,  minimum width = 4pt]
{
	(90-22.5:1) node {} -- (270-22.5:1) node {} -- (315-22.5:1) node {} -- (0-22.5:1) node {}
};
\draw[white, fill = white, minimum width = 4pt, inner sep = 0pt]
{
	(0-22.5:0) node {}
};
\draw[WildStrawberry]
{
	(45-22.5:1) -- (225-22.5:1)
};
\draw[white, fill = white, minimum width = 2pt, inner sep = 0pt]
{
	(0-22.5:0) node {}
};
\draw[black]
{
	(135-22.5:1) -- (315-22.5:1)
};
\end{tikzpicture}
&&
\bwstyle
\begin{tikzpicture}[baseline = 0, scale=0.75]
\draw[black]
{
	(180-22.5:1) -- 	(135-22.5:1) -- (90-22.5:1)
	(0-22.5:1) -- (315-22.5:1)
	(225-22.5:1) -- (180-22.5:1)
	(90-22.5:1) -- (270-22.5:1)
};
\draw[WildStrawberry, thick, minimum width = 4pt,]
{
	(225-22.5:1) node {} -- (270-22.5:1) node {} -- (315-22.5:1) node {}
	(135-22.5:1) node {}
};
\draw[LimeGreen, thick,  minimum width = 4pt]
{
	(180-22.5:1) node {} -- (0-22.5:1) node {} -- (45-22.5:1) node {} -- (90-22.5:1) node {}
};
\draw[white, fill = white, minimum width = 4pt, inner sep = 0pt]
{
	(0-22.5:0) node {}
};
\draw[black]
{
	(45-22.5:1) -- (225-22.5:1)
};
\draw[white, fill = white, minimum width = 2pt, inner sep = 0pt]
{
	(0-22.5:0) node {}
};
\draw[WildStrawberry]
{
	(135-22.5:1) -- (315-22.5:1)
};
\end{tikzpicture}
&&
\bwstyle
\begin{tikzpicture}[baseline = 0, scale=0.75]
\draw[black]
{
	(180-22.5:1) -- 	(135-22.5:1) -- (90-22.5:1)
	(0-22.5:1) -- (315-22.5:1)
	(225-22.5:1) -- (180-22.5:1)
	(90-22.5:1) -- (270-22.5:1)
};
\draw[LimeGreen, thick, minimum width = 4pt,]
{
	(225-22.5:1) node {} -- (270-22.5:1) node {} -- (315-22.5:1) node {}
	(135-22.5:1) node {}
};
\draw[WildStrawberry, thick,  minimum width = 4pt]
{
	(180-22.5:1) node {} -- (0-22.5:1) node {} -- (45-22.5:1) node {} -- (90-22.5:1) node {}
};
\draw[white, fill = white, minimum width = 4pt, inner sep = 0pt]
{
	(0-22.5:0) node {}
};
\draw[black]
{
	(45-22.5:1) -- (225-22.5:1)
};
\draw[white, fill = white, minimum width = 2pt, inner sep = 0pt]
{
	(0-22.5:0) node {}
};
\draw[LimeGreen]
{
	(135-22.5:1) -- (315-22.5:1)
};
\end{tikzpicture}
&&
\bwstyle
\begin{tikzpicture}[baseline = 0, scale=0.75]
\draw[black]
{
	(180-22.5:1) -- 	(135-22.5:1)
	(45-22.5:1) --
	(0-22.5:1) -- (315-22.5:1) --
	(270-22.5:1)	
	(90-22.5:1) -- (270-22.5:1)
};
\draw[WildStrawberry, thick, minimum width = 4pt,]
{
	(0-22.5:1) node {} -- (180-22.5:1) node {} -- (225-22.5:1) node {} -- (270-22.5:1) node {}
};
\draw[LimeGreen, thick,  minimum width = 4pt]
{
	(45-22.5:1) node {} -- (90-22.5:1) node {}  -- (135-22.5:1) node {}
	(315-22.5:1) node {}
};
\draw[white, fill = white, minimum width = 4pt, inner sep = 0pt]
{
	(0-22.5:0) node {}
};
\draw[black]
{
	(45-22.5:1) -- (225-22.5:1)
};
\draw[white, fill = white, minimum width = 2pt, inner sep = 0pt]
{
	(0-22.5:0) node {}
};
\draw[LimeGreen]
{
	(135-22.5:1) -- (315-22.5:1)
};
\end{tikzpicture}
&&
\bwstyle
\begin{tikzpicture}[baseline = 0, scale=0.75]
\draw[black]
{
	(180-22.5:1) -- 	(135-22.5:1)
	(45-22.5:1) --
	(0-22.5:1) -- (315-22.5:1) --
	(270-22.5:1)
	(90-22.5:1) -- (270-22.5:1)
};
\draw[LimeGreen, thick, minimum width = 4pt,]
{
	(0-22.5:1) node {} -- (180-22.5:1) node {} -- (225-22.5:1) node {} -- (270-22.5:1) node {}
};
\draw[WildStrawberry, thick,  minimum width = 4pt]
{
	(45-22.5:1) node {} -- (90-22.5:1) node {}  -- (135-22.5:1) node {}
	(315-22.5:1) node {}
};
\draw[white, fill = white, minimum width = 4pt, inner sep = 0pt]
{
	(0-22.5:0) node {}
};
\draw[black]
{
	(45-22.5:1) -- (225-22.5:1)
};
\draw[white, fill = white, minimum width = 2pt, inner sep = 0pt]
{
	(0-22.5:0) node {}
};
\draw[WildStrawberry]
{
	(135-22.5:1) -- (315-22.5:1)
};
\end{tikzpicture}
&&
\bwstyle
\begin{tikzpicture}[baseline = 0, scale=0.75]
\draw[black]
{
	(135-22.5:1) -- (90-22.5:1) -- (45-22.5:1) -- (0-22.5:1)
	(270-22.5:1) -- (225-22.5:1)
	(180-22.5:1) -- (0-22.5:1)
};
\draw[LimeGreen, thick, minimum width = 4pt,]
{
	(135-22.5:1) node {} -- (180-22.5:1) node {} -- (225-22.5:1) node {}
	(45-22.5:1) node {}
};
\draw[WildStrawberry, thick,  minimum width = 4pt]
{
	(90-22.5:1) node {} -- (270-22.5:1) node {} -- (315-22.5:1) node {} -- (0-22.5:1) node {}
};
\draw[white, fill = white, minimum width = 4pt, inner sep = 0pt]
{
	(0-22.5:0) node {}
};
\draw[LimeGreen]
{
	(45-22.5:1) -- (225-22.5:1)
};
\draw[white, fill = white, minimum width = 2pt, inner sep = 0pt]
{
	(0-22.5:0) node {}
};
\draw[black]
{
	(135-22.5:1) -- (315-22.5:1)
};
\end{tikzpicture}
&&
\bwstyle
\begin{tikzpicture}[baseline = 0, scale=0.75]
\draw[black]
{
	(90-22.5:1) -- (45-22.5:1)
	(315-22.5:1) --
	(270-22.5:1) -- (225-22.5:1)
	 -- (180-22.5:1)
	(180-22.5:1) -- (0-22.5:1)
};
\draw[WildStrawberry, thick, minimum width = 4pt,]
{
	(315-22.5:1) node {} -- (0-22.5:1) node {} -- (45-22.5:1) node {}
	(225-22.5:1) node {}
};
\draw[LimeGreen, thick,  minimum width = 4pt]
{
	(270-22.5:1) node {} -- (90-22.5:1) node {} -- (135-22.5:1) node {} -- (180-22.5:1) node {}
};
\draw[white, fill = white, minimum width = 4pt, inner sep = 0pt]
{
	(0-22.5:0) node {}
};
\draw[WildStrawberry]
{
	(45-22.5:1) -- (225-22.5:1)
};
\draw[white, fill = white, minimum width = 2pt, inner sep = 0pt]
{
	(0-22.5:0) node {}
};
\draw[black]
{
	(135-22.5:1) -- (315-22.5:1)
};
\end{tikzpicture}
&&
\bwstyle
\begin{tikzpicture}[baseline = 0, scale=0.75]
\draw[black]
{
	(90-22.5:1) -- (45-22.5:1)	(315-22.5:1) --
	(270-22.5:1) -- (225-22.5:1)
	 -- (180-22.5:1)
	(180-22.5:1) -- (0-22.5:1)
};
\draw[LimeGreen, thick, minimum width = 4pt,]
{
	(315-22.5:1) node {} -- (0-22.5:1) node {} -- (45-22.5:1) node {}
	(225-22.5:1) node {}
};
\draw[WildStrawberry, thick,  minimum width = 4pt]
{
	(270-22.5:1) node {} -- (90-22.5:1) node {} -- (135-22.5:1) node {} -- (180-22.5:1) node {}
};
\draw[white, fill = white, minimum width = 4pt, inner sep = 0pt]
{
	(0-22.5:0) node {}
};
\draw[LimeGreen]
{
	(45-22.5:1) -- (225-22.5:1)
};
\draw[white, fill = white, minimum width = 2pt, inner sep = 0pt]
{
	(0-22.5:0) node {}
};
\draw[black]
{
	(135-22.5:1) -- (315-22.5:1)
};
\end{tikzpicture}
\end{align*}
\begin{com} Possibly unify conventions for generators and states in pictures...
What do others think?
\end{com}

\subsection{Some high density bipartite examples}
Consider a bipartite graph $\Gamma(V,E)$ with bipartite structure $V=A\sqcup B$.
In many cases where there are sufficiently many edges,
there exists a legal system all of whose moves are either $A$ or $B$.
For instance the following easy criterion often applies in high density situations.
For instance, it applies to the graph $\Gamma$ in Figure~\ref{fig:TBWS}.

\begin{figure}
\begin{tikzpicture}[scale = 0.5]
\tikzstyle{every node}=[circle, draw, fill,
                        inner sep=0pt, minimum width = 3pt]
\foreach \x in {1,...,6}
{
\coordinate (a\x) at (-3,-\x);
\coordinate (b\x) at (3,-\x);
\node at (a\x) {};
\node at (b\x) {};
}
\draw (a1) -- (b1) -- (a2) -- (b2) -- (a3);
\draw (b2) -- (a3);
\draw (a1) -- (b5) -- (a2) -- (b6) -- (a3) -- (b4) -- (a4) -- (b5) -- (a5) -- (b6) -- (a6) -- (b3) -- (a5) -- (b2) -- (a4) -- (b1);
\draw (a2) -- (b3);
\draw [decoration={brace,amplitude=8pt},decorate] ($(b1)+(0.2,0.2)$) -- ($(b3)+(0.2,-0.2)$);
\draw [decoration={brace,amplitude=8pt},decorate] ($(b4)+(0.2,0.2)$) -- ($(b6)+(0.2,-0.2)$);
\draw [decoration={brace,amplitude=8pt,mirror},decorate] ($(a1)+(-0.2,0.2)$) -- ($(a3)+(-0.2,-0.2)$);
\draw [decoration={brace,amplitude=8pt,mirror},decorate] ($(a4)+(-0.2,0.2)$) -- ($(a6)+(-0.2,-0.2)$);
\tikzstyle{every node}=[]
\node at (-4.3,-2) {$A_1$};
\node at (-4.3,-5) {$A_2$};
\node at (4.3,-2) {$B_1$};
\node at (4.3,-5) {$B_2$};
\end{tikzpicture}
  \caption{\label{fig:TBWS}A bipartite graph $\Gamma$ with a legal system illustrating
  Lemma~\ref{lem:TBWS}.}
\end{figure}
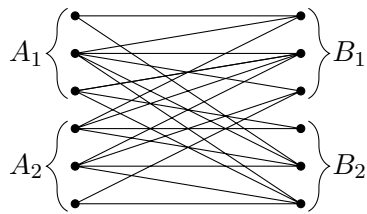

\begin{lem}\label{lem:TBWS}
Consider a bipartite graph $\Gamma$ whose vertex set has the bipartite structure  $A\sqcup B$
where $A = A_1\sqcup A_2$ and $B = B_1\sqcup B_2$.
If the following holds then each $A_i\sqcup B_j$ is a legal state for each $i,j$ and so $\{A,B\}$ is a legal system.
 \begin{enumerate}
\item\label{TBWS:1} For each $i,j$, and for each pair of vertices $a,a'\in A_i$, there is a path from $a$ to $a'$ in the subgraph induced by $A_i\sqcup B_j$.
\item\label{TBWS:2} For each $i,j$ each element of $B_j$ is adjacent to an element of $A_i$.
 \end{enumerate}
 \end{lem}
\begin{proof}
The orbit of $A_1 \cup B_1$ is
$\{A_i\cup B_j  \ : \  i,j \in \{1,2\} \ \}$.
For each state $A_i\cup B_j$, the vertices of $A_i$ lie in a single component by~\eqref{TBWS:1}
and each vertex of $B_j$ is joined to this component by~\eqref{TBWS:2}.
\end{proof}

 \begin{prob}\label{prob:4connected wins or subwins}
Is there a 4-connected finite graph $\Gamma$ with $\girth(\Gamma)\geq 4$ and $\curvature(\Gamma)\geq 0$ but no 4-connected (or even 3-connected) subgraph $\Gamma'$ with a legal system?
 \end{prob}
Figure~\ref{fig:Tutte} depicts  Tutte's graph which is 3-connected but contains no 3-connected subgraph with a legal system.
We are unable to make more than a vague connection to $n$-connectivity here but refer to Remark~\ref{rem:connectivity} and Theorem~\ref{thm:DWtoHam} and Conjecture~\ref{conj:reformed barnette}.

 \subsection{$24$-cell}
 The $24$-cell is one of six convex regular 4-polytopes. Its boundary is composed of $24$ octahedra with $6$ meeting at each vertex, and $3$ at each edge. It can be realized as a right-angled ideal hyperbolic polytope. Let $G$ be the Coxeter group of reflections in the 3-dimensional faces of the $24$-cell. Since the $24$-cell is self-dual $G = G(\Gamma)$ where $\Gamma$ is the $1$-skeleton of the $24$-cell.
 \subsubsection{The $1$-skeleton of the $24$-cell}
The graph $\Gamma$ is obtained from the $1$-skeleton of the $4$-cube as follows: for each of the $3$-cubes in the $4$-cube we add an \emph{extra} vertex and join it to all vertices of its ``surrounding'' $3$-cube.
    See the left graph in Figure~\ref{fig:24graph}.
The resulting graph has $24$ vertices: $16$ of the $4$-cube and $8$ extra vertices; and has $96$ edges: $32$ in the $4$-cube and $8$ edges for each extra vertex.
\subsubsection{Legal system in the $24$-cell}
 Consider the system of moves for $\Gamma$ corresponding to the $3$-coloring of $\Gamma$ where the $4$-cube is colored using a bipartite structure and
  all extra vertices are given a third color. Note that extra vertices are pairwise at distance $\geq 2$.

The state $S_b$ of the $4$-cube graph, presented on the right in Figure~\ref{fig:24graph}, has the property that its restriction to any $3$-cube is a legal state of the $3$-cube as in Example~\ref{ex:cube}. All states obtained from $S_b$ by the moves corresponding to the bipartite structure of the $4$-cube graph also have this property and remain legal. Thus the bipartition of the $4$-cube graph is a legal system (A similar argument works for an $n$-cube). We extend $S_b$ to a state $S_o$ of $\Gamma$. Since each extra vertex $v$ in $\Gamma$ is adjacent to all vertices of a $3$-cube, for each state $S$ in the orbit of $S_b$, the vertex $v$ is joined to both a vertex in $S$ and a vertex in $V-S$. Let $S_o$ be the union of $S_b$ and any subset of the set of extra vertices. Then $S_o$ provides a state for $\Gamma$ whose orbit is legal.
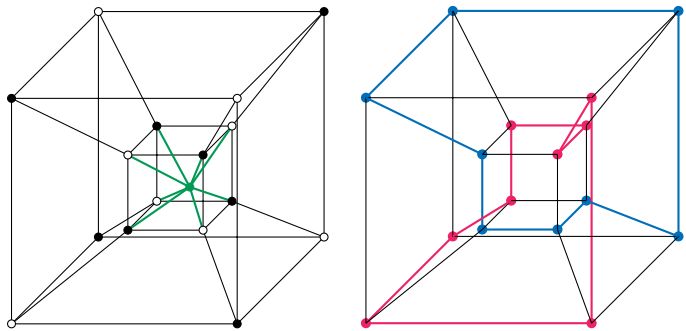
\begin{figure}
\begin{tikzpicture}[scale = 1]
\tikzstyle{every node}=[circle, draw, fill,
                        inner sep=1pt, minimum width = 3pt]
\coordinate (a) at (0,0,0);
\coordinate (b) at (0,0,1);
\coordinate (c) at (0,1,1);
\coordinate (d) at (0,1,0);
\coordinate (e) at (1,1,0);
\coordinate (f) at (1,0,0);
\coordinate (g) at (1,0,1);
\coordinate (h) at (1,1,1);

\coordinate (a') at (0,0,0);
\coordinate (b') at (0,0,3);
\coordinate (c') at (0,3,3);
\coordinate (d') at (0,3,0);
\coordinate (e') at (3,3,0);
\coordinate (f') at (3,0,0);
\coordinate (g') at (3,0,3);
\coordinate (h') at (3,3,3);

\coordinate (j) at (.65,.4,.55);

\node[draw,ForestGreen] at (j) {};
\draw
{
	(f) to (a) to (b) to (c) to (d) to (e) to (f) to (g) to (h) to (c)
	(a) to(d)
	(h) to (e)
	(b) to (g)
};

\foreach \x in {a,b,c,d,e,f,g,h}
{
	\draw[ForestGreen, thick] (\x) to (j);
}
\begin{scope}[every coordinate/.style={shift = {(-1.2,-0.9,-1.1)}}]
\draw[shift={(-1,-1,-1)}]
{
([c]f') to ([c]a') to ([c]b') to ([c]c') to ([c]d') to ([c]e') to ([c]f') to ([c]g') to ([c]h') to ([c]c')
	([c]a') to([c]d')
	([c]h') to ([c]e')
	([c]b') to ([c]g')
};
\foreach \x in {a,c,e,g}
{
	\draw (\x) to ([c]\x') node {};
}
\foreach \x in {b,d,f,h}
{
	\draw (\x) to ([c]\x') node[fill = white] {};
}
\end{scope}
\foreach \x in {a,c,e,g}
{
	\node[draw, fill = white] at (\x) {};
}
\foreach \x in {b,d,f,h}
{
	\node[draw] at (\x) {};
}
\end{tikzpicture}
\hspace{5pt}
\begin{tikzpicture}[scale = 1]
\tikzstyle{every node}=[circle, draw, fill,
                        inner sep=1pt, minimum width = 3pt]

\coordinate (a) at (0,0,0);
\coordinate (b) at (0,0,1);
\coordinate (c) at (0,1,1);
\coordinate (d) at (0,1,0);
\coordinate (e) at (1,1,0);
\coordinate (f) at (1,0,0);
\coordinate (g) at (1,0,1);
\coordinate (h) at (1,1,1);

\coordinate (a') at (0,0,0);
\coordinate (b') at (0,0,3);
\coordinate (c') at (0,3,3);
\coordinate (d') at (0,3,0);
\coordinate (e') at (3,3,0);
\coordinate (f') at (3,0,0);
\coordinate (g') at (3,0,3);
\coordinate (h') at (3,3,3);

\coordinate (j) at (.55,.4,.55);

\foreach \x in {a,b,c,d,e,f,g,h}
{
	\node[draw] at (\x) {};
}
\draw[WildStrawberry, thick]
{
	(h) node {} to (e) node {} to (d) node {} to (a) node {}
};
\draw[NavyBlue, thick]
{
	(c) node {} to (b) node {} to (g) node {} to (f) node {}
};
\draw
{
	(b) to (a) to (f) to (e)
	(d) to (c) to (h) to (g)
};
\begin{scope}[every coordinate/.style={shift = {(-1.2,-0.9,-1.1)}}]
\draw[WildStrawberry, thick]
{
	([c]a') node {} to ([c]b') node {} to ([c]g') node {} to ([c]h') node {}
};
\draw[NavyBlue, thick]
{
	([c]c') node {} to ([c]d') node {} to ([c]e') node {} to ([c]f') node {}
};
\draw
{
	([c]d') to ([c]a') to ([c]f') to ([c]g')
	([c]b') to ([c]c') to ([c]h') to ([c]e')
};
\foreach \x in {h,a}
{
	\draw[WildStrawberry,thick] (\x) to ([c]\x');
}
\foreach \x in {c,f}
{
	\draw[NavyBlue,thick] (\x) to ([c]\x');
}
\foreach \x in {b,d,e,g}
{
	\draw (\x) to ([c]\x');
}
\end{scope}
\end{tikzpicture}
\caption{The $1$-skeleton of the $24$-cell has $8$ green vertices corresponding to $3$-dimensional faces of the $4$-cube. We illustrate only one such vertex here. On the right there is a legal state for the $1$-skeleton of the $4$-cube. All elements of its orbit are legal where the system of moves correspond to bipartition of the $1$-skeleton of the $4$-cube.}
\label{fig:24graph}
\end{figure}

 \subsection{Icosahedron}\label{sub:icosahedron}
Consider the $6$-coloring of the icosahedron on the left in Figure~\ref{fig:icosahedron}.

For each vertex $v\in V$ there exists exactly one color \emph{k} such that no neighbor of $v$ has it. In particular, the set of elements in color $k$ or in the color of $v$ is a move at $v$, e.g. the pink and purple vertices form a move at the top pink vertex.
We will show that the system $M$ of moves of this form is legal. This system generates an index $2$ subgroup of a group $M'$ generated by all single color sets. Consider the collection $\mathcal S$ of states having exactly one vertex of each color. The collection $\mathcal S$ is a single orbit under $M'$-action. Not all such states are legal, but all nonlegal state of such form lie in a single orbit under $M$-action.
To see that note that for any nonlegal $S\in\mathcal S$ the graph induced by either $S$ or $V-S$ consists of two disjoint triangles as in the picture below. Indeed, the connected component of this graph cannot be a single vertex or a pair of vertices connected by an edge since every vertex and every such pair have two neighbors in the same color. Any two illegal states differ on an even number of colors.

\begin{figure}
\tikzstyle{every node}=[circle, draw, fill,
                        inner sep=2pt, minimum width = 4pt]
 \begin{tikzpicture}[scale=1]
\draw
{
	(90:.5) to (210:.5) to (330:.5) to (90:.5)
	
	(30:1) to (90:.5) to (150:1) to (210:.5) to (270:1) to (330:.5) to (30:1)
	
	(90:1.5) to (30:1) to (330:1.5) to (270:1) to (210:1.5) to (150:1) to (90:1.5) to (90:.5)
	(330:1.5) to (330:.5)
	(210:1.5) to (210:.5)
	
	(30:2) to (90:1.5) to (150:2) to (210:1.5) to (270:2) to (330:1.5) to (30:2) to (30:1)
	(150:2) to (150:1)
	(270:2) to (270:1)
	
	(30:2) to[out=120 , in=60] (150:2)
	(150:2) to[out=240, in=180] (270:2)
	(270:2) to[out=0,in=300] (30:2)
};

\draw[WildStrawberry,dotted, thick]
{
	(90:.5) node {} to (270:1) node {}	
};
\draw[MidnightBlue, dotted, thick]
{
	(330:.5) node {} to[out = 30, in = 240] (30:2) node {}
};
\draw[PineGreen,dotted, thick]
{
	(210:.5) node {} to[out = 150, in = 300] (150:2) node {}
};
\draw[Melon,dotted, thick]
{
	(30:1) node {} to[out = 150, in = 30] (150:1) node {}
};
\draw[Purple, dotted, thick]
{
	(210:1.5) node {} to[out = 320, in =220] (330:1.5) node {}
};
\draw[LimeGreen,dotted, thick]
{
	(270:2) node {} to (270:2.5)
	(90:1.5) node {} to (90:2.2)
};
\end{tikzpicture}
 \begin{tikzpicture}[scale=1]
\draw
{
	(90:.5) to (210:.5) to (330:.5) to (90:.5)
	
	(30:1) to (90:.5) to (150:1) to (210:.5) to (270:1) to (330:.5) to (30:1)
	
	(90:1.5) to (30:1) to (330:1.5) to (270:1) to (210:1.5) to (150:1) to (90:1.5) to (90:.5)
	(330:1.5) to (330:.5)
	(210:1.5) to (210:.5)
	
	(30:2) to (90:1.5) to (150:2) to (210:1.5) to (270:2) to (330:1.5) to (30:2) to (30:1)
	(150:2) to (150:1)
	(270:2) to (270:1)
	
	(30:2) to[out=120 , in=60] (150:2)
	(150:2) to[out=240, in=180] (270:2)
	(270:2) to[out=0,in=300] (30:2)
};

\draw[WildStrawberry,dotted, thick]
{
	(90:.5) node {} to (270:1) node {}	
};
\draw[MidnightBlue, dotted, thick]
{
	(330:.5) node {} to[out = 30, in = 240] (30:2) node {}
};
\draw[PineGreen,dotted, thick]
{
	(210:.5) node {} to[out = 150, in = 300] (150:2) node {}
};
\draw[Melon,dotted, thick]
{
	(30:1) node {} to[out = 150, in = 30] (150:1) node {}
};
\draw[Purple, dotted, thick]
{
	(210:1.5) node {} to[out = 320, in =220] (330:1.5) node {}
};
\draw[LimeGreen,dotted, thick]
{
	(270:2) node {} to (270:2.5)
	(90:1.5) node {} to (90:2.2)
};
\draw[black]
{
	(90:.5) circle (5pt)
	(330:.5) circle (5pt)
	(30:1) circle (5pt)
	(210:1.5) circle (5pt)
	(210:.5) circle (5pt)
	(270:2) circle (5pt)
};
\end{tikzpicture}
 \begin{tikzpicture}[scale=1]
\draw
{
	(90:.5) to (210:.5) to (330:.5) to (90:.5)
	
	(30:1) to (90:.5) to (150:1) to (210:.5) to (270:1) to (330:.5) to (30:1)
	
	(90:1.5) to (30:1) to (330:1.5) to (270:1) to (210:1.5) to (150:1) to (90:1.5) to (90:.5)
	(330:1.5) to (330:.5)
	(210:1.5) to (210:.5)
	
	(30:2) to (90:1.5) to (150:2) to (210:1.5) to (270:2) to (330:1.5) to (30:2) to (30:1)
	(150:2) to (150:1)
	(270:2) to (270:1)
	
	(30:2) to[out=120 , in=60] (150:2)
	(150:2) to[out=240, in=180] (270:2)
	(270:2) to[out=0,in=300] (30:2)
};

\draw[WildStrawberry,dotted, thick]
{
	(90:.5) node {} to (270:1) node {}	
};
\draw[MidnightBlue, dotted, thick]
{
	(330:.5) node {} to[out = 30, in = 240] (30:2) node {}
};
\draw[PineGreen,dotted, thick]
{
	(210:.5) node {} to[out = 150, in = 300] (150:2) node {}
};
\draw[Melon,dotted, thick]
{
	(30:1) node {} to[out = 150, in = 30] (150:1) node {}
};
\draw[Purple, dotted, thick]
{
	(210:1.5) node {} to[out = 320, in =220] (330:1.5) node {}
};
\draw[LimeGreen,dotted, thick]
{
	(270:2) node {} to (270:2.5)
	(90:1.5) node {} to (90:2.2)
};
\draw[black]
{
	(90:.5) circle (5pt)
	(330:.5) circle (5pt)
	(30:1) circle (5pt)
	(210:1.5) circle (5pt)
	(150:2) circle (5pt)
	(270:2) circle (5pt)
};
\end{tikzpicture}
\caption{The move at $v$ consists of the two vertices with the same color as $v$
together with the two vertices with the color not adjacent to $v$.
Let $S_o$ be a state consisting of exactly one vertex of each color. There are two orbits of such states.
One is a legal system and is illustrated in the middle. On the right there is a non-legal state. We use a different convention here than in previous examples: vertices of the same color correspond to moves, and circled vertices form a state.}\
\label{fig:icosahedron}
\end{figure}
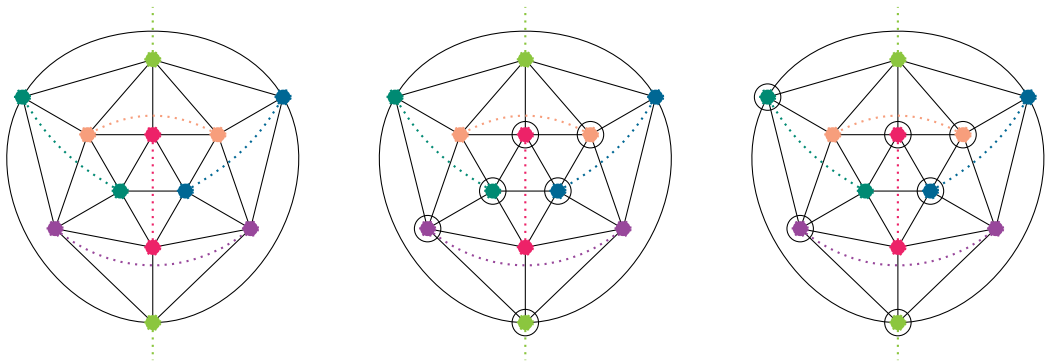

\subsection{$600$-cell}
The $120$-cell is a notable convex regular $4$-polytope that admits an embedding in $\hyperbolic^4$ as  a right-angled polytope. Identifying opposite $3$-dimensional faces of the $120$-cell gives a compact hyperbolic $4$-manifold \cite{Davis85}. The dual of the $120$-cell is the $600$-cell whose boundary is a flag complex that is built  of $600$ tetrahedra and is characterized by the property that the link of each vertex is an icosahedron. The reflection group associated to the $120$-cell generated by reflections along its $3$-dimensional faces is the Coxeter group whose defining graph is the $1$-skeleton of the $600$-cell.
We refer to \cite{wiki:600cell} for a description of the 600-cell, and note that the construction below was motivated by the discussion there.

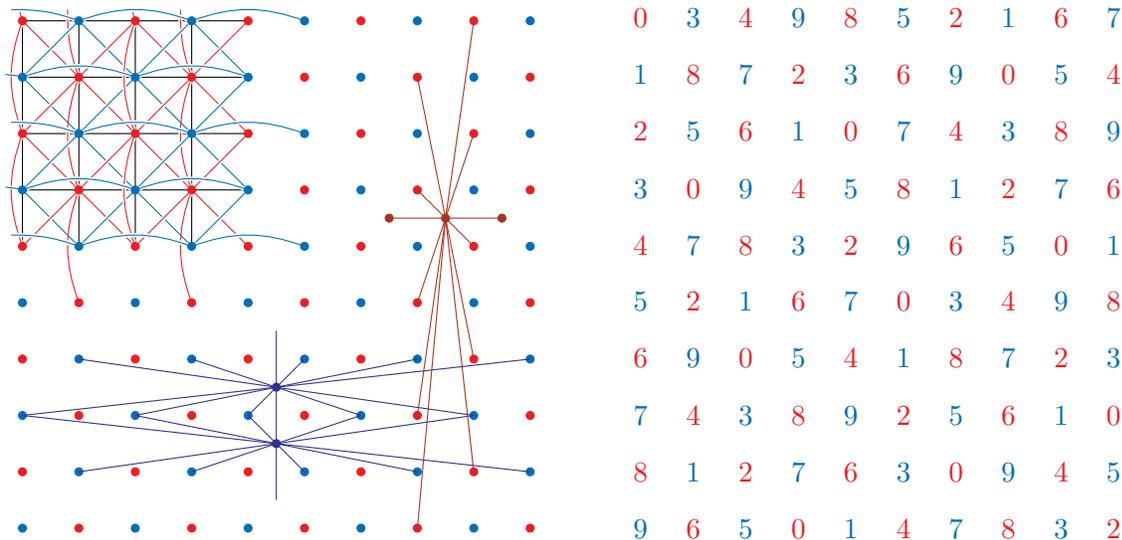
\begin{figure}
\[
\begin{tikzpicture}[scale= 0.75, baseline = -11]
\tikzstyle{every node}=[circle, draw, fill,
                        inner sep=1pt, minimum width = 3pt]

\foreach \s in {0,1,2,3}
{
	\foreach \t in {5,6,7,8}
	{
		\draw (\s,\t) to (\s,\t+1);
		\draw (\s,\t+1) to (\s+1,\t+1);
	}
}
\foreach \s in {1,3}
{
	\draw[Red] (\s,6) to (\s+1,5);
	\draw[Red] (\s,8) to (\s+1,7);
	\draw[white, ultra thick] (\s,5) to (\s+1,6);
	\draw[NavyBlue] (\s,5) to (\s+1,6);
	\draw[white, ultra thick] (\s,7) to (\s+1,8);
	\draw[NavyBlue] (\s,7) to (\s+1,8);
	\draw[NavyBlue] (\s,7) to (\s+1,6);
	\draw[NavyBlue] (\s,9) to (\s+1,8);
	\draw[white, ultra thick] (\s,6) to (\s+1,7);
	\draw[Red] (\s,6) to (\s+1,7);
	\draw[white, ultra thick] (\s,8) to (\s+1,9);
	\draw[Red] (\s,8) to (\s+1,9);
	\draw[white, ultra thick] (\s,4) to[out=110, in = 250] (\s,6) to[out=110, in = 250] (\s,8)  to[out=110, in = 270] (\s-0.2,9.2);
	\draw[Red] (\s,4) to[out=110, in = 250] (\s,6) to[out=110, in = 250] (\s,8)  to[out=110, in = 270] (\s-0.2,9.2);
}
\foreach \s in {0,2}
{
	\draw[Red] (\s,7) to (\s+1,6);
	\draw[Red] (\s,9) to (\s+1,8);
	\draw[white, ultra thick] (\s,6) to (\s+1,7);
	\draw[NavyBlue] (\s,6) to (\s+1,7);
	\draw[white, ultra thick] (\s,8) to (\s+1,9);
	\draw[NavyBlue] (\s,8) to (\s+1,9);
	\draw[NavyBlue] (\s,6) to (\s+1,5);
	\draw[NavyBlue] (\s,8) to (\s+1,7);
	\draw[white, ultra thick] (\s,5) to (\s+1,6);
	\draw[Red] (\s,5) to (\s+1,6);
	\draw[white, ultra thick] (\s,7) to (\s+1,8);
	\draw[Red] (\s,7) to (\s+1,8);
	\draw[white, ultra thick] (\s,5) to[out=110, in = 250] (\s,7)  to[out=110, in = 250] (\s,9);
	\draw[Red] (\s,5) to[out=110, in = 250] (\s,7)  to[out=110, in = 250] (\s,9) to [out=110, in=280] (\s-.05,9.2);

}
\foreach \t in {5,7,9}
{
	\draw[white, ultra thick] (5,\t) to[out=160,in=20] (3,\t)  to[out=160,in=20] (1,\t) to[out=160,in =0] (-.2,\t+.2);
	\draw[NavyBlue] (5,\t) to[out=160,in=20] (3,\t)  to[out=160,in=20] (1,\t) to[out=160,in =0] (-.2,\t+.2);
}
\foreach \t in {6,8}
{
	\draw[white, ultra thick] (4,\t) to[out=160,in=20] (2,\t)  to[out=160,in=20] (0,\t) to[out = 160, in = -10] (-.2,\t+0.05);
	\draw[NavyBlue] (4,\t) to[out=160,in=20] (2,\t)  to[out=160,in=20] (0,\t) to[out = 160, in = -10] (-.3,\t+0.05);
}

\foreach \s in {0,2,4,6,8}
{
	\foreach \t in {0,2,4,6,8}
	{
		\node[draw,NavyBlue] at (\s,\t) {};
	}
	\foreach \t in {1,3,5,7,9}
	{
		\node[Red] at (\s,\t) {};
	}
}
\foreach \s in {1,3,5,7,9}
{
	\foreach \t in {0,2,4,6,8}
	{
		\node[Red] at (\s,\t) {};
	}
	\foreach \t in {1,3,5,7,9}
	{
		\node[NavyBlue] at (\s,\t) {};
	}
}

\draw[Blue] (4.5,3.5) to (4.5,2.5) node {} to (4.5,1.5) node {} to (4.5,0.5);
\foreach \s in {0,2,4,6,8}
{
	\draw[Blue] (\s,2) to (4.5,1.5);
	\draw[Blue] (\s+1,1) to (4.5,1.5);
}
\foreach \s in {0,2,4,6,8}
{
	\draw[Blue] (\s,2) to (4.5,2.5);
	\draw[Blue] (\s+1,3) to (4.5,2.5);
}

\draw[Mahogany] (6.5,5.5) node {} to (7.5,5.5) node {} to (8.5,5.5) node {} ;
\foreach \t in {1,3,5,7,9}
{
	\draw[Mahogany] (7,\t-1) to (7.5,5.5);
	\draw[Mahogany] (8,\t) to (7.5,5.5);
}
\end{tikzpicture}
\hspace{30pt}
\begin{tikzpicture}[row sep = 0.25cm, column sep =0.25cm]
\matrix(m)
{
\node[Red]{
0}; & \node[NavyBlue]{3}; & \node[Red]{4}; & \node[NavyBlue]{9}; & \node[Red]{8}; & \node[NavyBlue]{5}; & \node[Red]{2}; & \node[NavyBlue]{1}; & \node[Red]{6}; & \node[NavyBlue]{7}; \\ \node[NavyBlue]{
1}; & \node[Red]{8}; & \node[NavyBlue]{7}; & \node[Red]{2}; & \node[NavyBlue]{3}; & \node[Red]{6}; & \node[NavyBlue]{9}; & \node[Red]{0}; & \node[NavyBlue]{5}; & \node[Red]{4}; \\ \node[Red]{
2}; & \node[NavyBlue]{5}; & \node[Red]{6}; & \node[NavyBlue]{1}; & \node[Red]{0}; & \node[NavyBlue]{7}; & \node[Red]{4}; & \node[NavyBlue]{3}; & \node[Red]{8}; & \node[NavyBlue]{9}; \\ \node[NavyBlue]{
3}; & \node[Red]{0}; & \node[NavyBlue]{9}; & \node[Red]{4}; & \node[NavyBlue]{5}; & \node[Red]{8}; & \node[NavyBlue]{1}; & \node[Red]{2}; & \node[NavyBlue]{7}; & \node[Red]{6}; \\ \node[Red]{
4}; & \node[NavyBlue]{7}; & \node[Red]{8}; & \node[NavyBlue]{3}; & \node[Red]{2}; & \node[NavyBlue]{9}; & \node[Red]{6}; & \node[NavyBlue]{5}; & \node[Red]{0}; & \node[NavyBlue]{1}; \\ \node[NavyBlue]{
5}; & \node[Red]{2}; & \node[NavyBlue]{1}; & \node[Red]{6}; & \node[NavyBlue]{7}; & \node[Red]{0}; & \node[NavyBlue]{3}; & \node[Red]{4}; & \node[NavyBlue]{9}; & \node[Red]{8}; \\ \node[Red]{
6}; & \node[NavyBlue]{9}; & \node[Red]{0}; & \node[NavyBlue]{5}; & \node[Red]{4}; & \node[NavyBlue]{1}; & \node[Red]{8}; & \node[NavyBlue]{7}; & \node[Red]{2}; & \node[NavyBlue]{3}; \\ \node[NavyBlue]{
7}; & \node[Red]{4}; & \node[NavyBlue]{3}; & \node[Red]{8}; & \node[NavyBlue]{9}; & \node[Red]{2}; & \node[NavyBlue]{5}; & \node[Red]{6}; & \node[NavyBlue]{1}; & \node[Red]{0}; \\ \node[Red]{
8}; & \node[NavyBlue]{1}; & \node[Red]{2}; & \node[NavyBlue]{7}; & \node[Red]{6}; & \node[NavyBlue]{3}; & \node[Red]{0}; & \node[NavyBlue]{9}; & \node[Red]{4}; & \node[NavyBlue]{5}; \\ \node[NavyBlue]{
9}; & \node[Red]{6}; & \node[NavyBlue]{5}; & \node[Red]{0}; & \node[NavyBlue]{1}; & \node[Red]{4}; & \node[NavyBlue]{7}; & \node[Red]{8}; & \node[NavyBlue]{3}; & \node[Red]{2};\\
};
\end{tikzpicture}
\]
\caption{The $1$-skeleton of the $600$-cell is on the left. Red and blue vertices are even and odd respectively. One even hovering vertex and two consecutive odd hovering vertices are illustrated. On the right there is a legal system. Moves correspond to all vertices labelled with the same number. There are also moves corresponding to the hovering vertices which are not illustrated in the figure.}
\label{fig:600graph}
\end{figure}

\subsubsection{The 1-skeleton of the $600$-cell}
We refer to Figure~\ref{fig:600graph}.
Begin with a torus represented by a $10\times 10$ grid.
Note that it has a bipartite structure and we refer to its two classes of vertices as \emph{even} and \emph{odd}.
(We regard the northwest most vertex in the grid as even.)

We now add the following edges which join odd-to-odd and even-to-even:
For each of the 100 squares of the torus, we add both diagonals as edges.
For each even vertex we add an edge joining it to the two vertices that are two edges north and south of it in its column.
For each odd vertex we add an edge joining it to the two vertices that are two edges east and west of it in its row.
We use the terms \emph{row-cycle} and \emph{column-cycle} to refer to the corresponding cycles within the torus.

We now add 20 further \emph{hovering} vertices. Ten of these are \emph{odd} and ten are \emph{even}.
Each even hovering vertex corresponds to a consecutive pair of column-cycles and we attach it to all ten even vertices in those column-cycles. We add an edge between consecutive even hovering vertices.
Similarly, each odd hovering vertex corresponds to a consecutive pair of row-cycles and we attach it to all odd vertices in those row-cycles. We add an edge between consecutive odd hovering vertices.
The resulting graph $\Gamma$ has 120 vertices, and 720 edges.

Observe that the automorphism group of $\Gamma$
acts transitively on the non-hovering vertices and also on the hovering vertices.
It therefore suffices to examine a vertex $v$ in each case
and confirm that the graph spanned by the vertices adjacent to $v$ is an icosahedral graph
as illustrated in Figure~\ref{fig:link in 600}.
A blue hovering vertex of $\Gamma$ is adjacent to two horizontal row-cycles of blue vertices as well as two blue hovering vertices. Its link corresponds to an octahedron obtained by subdividing a 5-sided antiprism: Its horizontal row-cycles yield 5-cycles in the link that are connected by an alternating band of triangles, and each hovering vertex in the link is adjacent to all the vertices of one of these 5-cycles.
A blue non-hovering vertex is adjacent to four red non-hovering vertices, six blue non-hovering vertices, and two blue hovering vertices. Its link corresponds to an icosahedron obtained by subdividing a cuboctahedron.

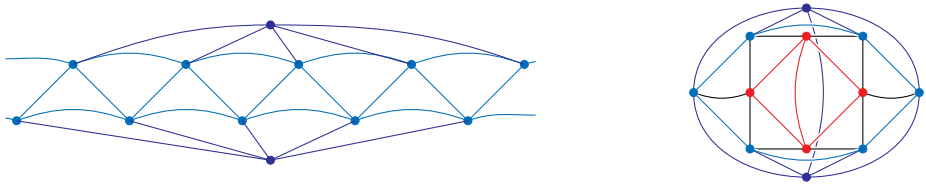
\begin{figure}\[
\begin{tikzpicture}[scale= 0.75, baseline = -11]
\tikzstyle{every node}=[circle, draw, fill,
                        inner sep=1pt, minimum width = 3pt]
\begin{scope}[shift={(0,-0.5)}]
\foreach \s in {0,2,4,6}
{
	\draw[NavyBlue] (\s,0) node {} to (\s+1,1) node {} to (\s+2,0) to[out=160, in = 20] (\s,0);
	\draw[NavyBlue] (\s+1,1) to[out=20, in = 160] (\s+3,1);
}
\draw[NavyBlue] (8,0) node {} to (9,1) node {};
\draw[NavyBlue] (0,0) to[out = 160, in = -10] (-.2,0.05);
\draw[NavyBlue] (1,1) to[out=160,in =0] (-.2,1.1);
\draw[NavyBlue] (8,0) to[out=20,in =180] (9.2,0.1);
\draw[NavyBlue] (9,1) to[out = 20, in = 190] (9.2,1.05);
\draw[Blue] (4.5,-0.7) node {};
\foreach \s in {0,2,4,6,8}
{
	\draw[Blue] (\s,0) to (4.5,-0.7);
}
\draw[Blue] (4.5,1.7) node {};
\foreach \s in {3,5,7}
{
	\draw[Blue] (\s,1) to (4.5,1.7);
}
\draw[Blue] (1,1) to[out = 20, in=180] (4.5,1.7);
\draw[Blue] (9,1) to[out=160, in =0] (4.5,1.7);
\end{scope}

\begin{scope}[shift={(14,0)}]
\draw[Blue] (0,-1.5) to[out=70,in=290] (0,1.5);

\draw[white,ultra thick] {
(-1,0) to (-1,1)
(-1,0) to (-1,-1)
(-1,0) to[out=200, in=-20] (-2,0)
(1,0) to[out=-20, in=200] (2,0)
(1,0) to (1,1)
(1,0) to (1,-1)
(0,1) to (-1,1)
(0,1) to (1,1)
(0,-1) to (-1,-1)
(0,-1) to (1,-1)
};

\draw {
(-1,0) to (-1,1)
(-1,0) to (-1,-1)
(-1,0) to[out=200, in=-20] (-2,0)
(1,0) to[out=-20, in=200] (2,0)
(1,0) to (1,1)
(1,0) to (1,-1)
(0,1) to (-1,1)
(0,1) to (1,1)
(0,-1) to (-1,-1)
(0,-1) to (1,-1)
};
\draw[white, ultra thick] (-1,-1) to (-2,0) to (-1,1) to[out=20, in=160] (1,1) to (2,0) to (1,-1)  to[out=200, in=-20] (-1,-1);
\draw[white, ultra thick] (-1,0) to (0,1) to (1,0) to (0,-1) to (-1,0);
\draw[white, ultra thick] (0,-1) to[out=110, in = 250] (0,1);
\draw[Red] (-1,0) node {} to (0,1) node {} to (1,0) node {} to (0,-1) node {} to (-1,0);
\draw[Red] (0,-1) to[out=110, in = 250] (0,1);
\draw[Blue] {
(0,-1.5) node {} to(-1,-1)
(0,-1.5) to (1,-1)
(0,-1.5) to[out=180, in=270] (-2,0)
(0,-1.5) to[out=0, in=270] (2,0)
};
\draw[Blue] {
(0,1.5) node {} to(-1,1)
(0,1.5) to (1,1)
(0,1.5) to[out=180, in=90] (-2,0)
(0,1.5) to[out=0, in=90] (2,0)
};
\draw[NavyBlue] (-1,-1) node {} to (-2,0) node {} to (-1,1) node {} to[out=20, in=160] (1,1) node {} to (2,0) node {} to (1,-1) node {} to[out=200, in=-20] (-1,-1);

\end{scope}
\end{tikzpicture}\]
\caption{\label{fig:link in 600}Links of a blue hovering vertex and non-hovering vertex.}
\end{figure}

\subsubsection{A legal system for the 600-cell 1-skeleton}

We color the vertices of the grid-vertices of $\Gamma$ as on the right in Figure~\ref{fig:600graph}:
The vertices in the westmost column are colored $0,1,\ldots, 9$ starting from the top of the grid.
The subsequent columns are colored so that even numbers decrease traveling southeast,
and odd numbers increase traveling northeast.

The cycle of even hovering vertices are colored $0',2', 4',6', 8', 0',2', 4',6',8'$.
The cycle of odd hovering vertices are colored $1', 3', 5', 7', 9', 1', 3', 5', 7', 9'$.

The start state $S_o$ consists of all vertices in alternate row-cycles of the torus,
together with a choice of alternate vertices within each hovering cycle.

Note that there are exactly 20 moves, corresponding to the distinct colors.

\subsubsection{Proof that it is a legal system}

Consider a state $S$ in the orbit of $S_o$.
As the reader can verify, a feature of the colored moves is that for consecutive column-cycles, their odd vertices in $S$ are \emph{complementary} in the sense that that for each color, exactly one of the two vertices with that color lies in $S$. The analogous statement holds for even vertices and row-cycles.

There is a \emph{degenerate} odd case where alternate column-cycles have all their odd vertices in $S$.
 Similarly there is a \emph{degenerate} even case where alternate row-cycles have all their even vertices in $S$.

There are now four cases to consider according to whether odd or even are degenerate.

We first observe that in a doubly-degenerate case, (e.g. $S_o$), the grid-edges connect each row-cycle to each column-cycle, and so $S$ is connected since in addition, each hovering vertex in $S$ is automatically connected at one of its two sides.

We now examine a non-degenerate case where
 some but not all odd vertices in a column-cycle are in $S$.
We will show that the odd vertices in $S$ are connected in this case:
The consecutive complementary color property shows that every odd vertex is connected to a ``traversing-cycle'' in the odd part of $S$ which contains a vertex in each column-cycle.
Indeed, consider a northmost odd vertex of $S$ within $S$-component of a column-cycle, as the odd vertex north of it is not in $S$,
we see that the vertex southeast of it is in $S$, and it is thus connected to an odd vertex in the next column eastward. Likewise, is connected to a traversing-cycle intersecting each column-cycle.
Each odd hovering $S$-vertex is connected to each traversing-cycle so we are done.

An analogous argument shows that for a non-degenerate even case, where some but not all even vertices in a row-cycle are in $S$, there is a traversing-cycle consisting of even vertices of $S$ that contains a vertex from each row-cycle, and contains a vertex from each column-cycle.

If the odd vertices of $S$ are degenerate but the even are not, then the traversing cycle in the even part of $S$
is connected to all odd rows. The case where the even part is degenerate and the odd part is connected analogously.

We now consider the case where neither odd nor even is degenerate:
Consider a column-cycle with a maximal number of vertices in $S$.
If it has four then these vertices connect to all even vertices and so the odd part is connected to the even part.
If it has three then these must be consecutive, since otherwise separate groups of two and one will be adjacent to all even vertices in that column. The even and odd are thus connected unless there is only one even state in $S$ within that column. But then the complementary property assures that there are four even vertices in $S$ within the next column. But these are connected to all odd vertices within that column.

\begin{rem}
The group $G(\Gamma)$ is a cocompact right-angled reflection group in $\hyperbolic^4$
that was studied by Bowditch and Mess \cite{BowditchMess94}  who showed that $G(\Gamma)$ is  \emph{incoherent} in the sense that it has a f.g. subgroup that is not f.p. and we refer also to the work of M.Kapovich and Potyagailo and Vinberg \cite{KapovichPotyagailoVinberg2008} as well as
\cite{Kapovich2013}.
\end{rem}

\subsection{Brinkmann graph}

The Brinkmann graph is a very symmetric graph with $21$ vertices and $42$ edges, and with $\girth(\Gamma)=5$. Thus $\curvature(\Gamma)=1$.
We describe a legal system for $\Gamma$ associated to a coloring in Figure~\ref{fig:brinkmann}.
The Brinkmann graph obviously has subgraphs $\Gamma'$ with $\curvature(\Gamma')=0$,
but we have not checked if any such subgraph has a legal system.

\begin{figure}\label{fig:brinkmann}\centering
\tikzstyle{every node}=[circle, draw, fill,
                        inner sep=0pt, minimum width = 3pt]
\begin{tikzpicture}[scale = 1]
\begin{scope}
\tikzstyle{every node}=[]
\draw [->, color=NavyBlue, very thick] (2.9,-1.4) -- node[below] {} (1.8,-1.4);
\draw [->, very thick] (4,-1.9) -- node[right] {} (4,-2.6);
\draw [->, color=ForestGreen, very thick] (5,-1.8) -- node[right] {} (6.5,-3.3);
\draw [->, color=WildStrawberry, very thick] (5.8,-1.4) -- node[below] {} (7,-1.4);
\draw [->, color=NavyBlue, very thick] (9.6,-1.4) -- node[below] {} (10.7,-1.4);
\draw [->, color=NavyBlue, very thick] (9.6,-3.1) -- node[above] {} (10.7,-3.1);
\draw [->, color=NavyBlue, very thick] (2.9,-3.1) -- node[above] {} (1.8,-3.1);
\end{scope}
\begin{scope}[shift={(0,0)}]
\foreach \x in {0,...,6}
{
	\coordinate (s\x) at (310 + \x*360/7 : 2);
}

\tkzInterLL(s0,s3)(s1,s5) \tkzGetPoint{t4}
\tkzInterLL(s1,s4)(s2,s6) \tkzGetPoint{t5}
\tkzInterLL(s2,s5)(s3,s0) \tkzGetPoint{t6}
\tkzInterLL(s3,s6)(s4,s1) \tkzGetPoint{t0}
\tkzInterLL(s4,s0)(s5,s2) \tkzGetPoint{t1}
\tkzInterLL(s5,s1)(s6,s3) \tkzGetPoint{t2}
\tkzInterLL(s6,s2)(s0,s4) \tkzGetPoint{t3}

 \draw (s0) -- (s3);
 \draw (s1) -- (s4);
 \draw (s2) -- (s5);
 \draw (s3) -- (s6);
 \draw (s4) -- (s0);
 \draw (s5) -- (s1);
 \draw (s6) -- (s2);
 \draw (s0) -- (s2);
 \draw (s1) -- (s3);
 \draw (s2) -- (s4);
 \draw (s3) -- (s5);
 \draw (s4) -- (s6);
 \draw (s5) -- (s0);
 \draw (s6) -- (s1);

\draw[ForestGreen] (s0) node {};
\draw[ForestGreen] (s1) node {};
\draw[NavyBlue] (s2) node {};
\draw[NavyBlue] (s3) node {};
\draw[WildStrawberry] (s4) node {};
\draw[WildStrawberry] (s5) node {};
\draw (s6) node {};

\tkzInterLL(s0,t0)(t6,t1) \tkzGetPoint{r0}
\tkzInterLL(s1,t1)(t0,t2) \tkzGetPoint{r1}
\tkzInterLL(s2,t2)(t1,t3) \tkzGetPoint{r2}
\tkzInterLL(s3,t3)(t2,t4) \tkzGetPoint{r3}
\tkzInterLL(s4,t4)(t3,t5) \tkzGetPoint{r4}
\tkzInterLL(s5,t5)(t4,t6) \tkzGetPoint{r5}
\tkzInterLL(s6,t6)(t5,t0) \tkzGetPoint{r6}

\draw[ForestGreen] (t0) node {};
\draw[NavyBlue] (t1) node {};
\draw[NavyBlue] (t2) node {};
\draw[WildStrawberry] (t3) node {};
\draw[WildStrawberry] (t4) node {};
\draw (t5) node {};
\draw[ForestGreen] (t6) node {};

 \draw (r0) -- (r3);
 \draw (r1) -- (r4);
 \draw (r2) -- (r5);
 \draw (r3) -- (r6);
 \draw (r4) -- (r0);
 \draw (r5) -- (r1);
 \draw (r6) -- (r2);

\draw[WildStrawberry] (r0) node {};
\draw (r1) node {};
\draw[ForestGreen] (r2) node {};
\draw[ForestGreen] (r3) node {};
\draw[NavyBlue] (r4) node {};
\draw[NavyBlue] (r5) node {};
\draw[WildStrawberry] (r6) node {};

\draw[black]
{
	(s1) circle (4pt)
	(s2) circle (4pt)
	(s4) circle (4pt)
	(s6) circle (4pt)
	(t2) circle (4pt)
	(t4) circle (4pt)
	(t6) circle (4pt)
	(r0) circle (4pt)
	(r2) circle (4pt)
	(r5) circle (4pt)
};
\end{scope}
\begin{scope}[shift = {(4,0)}]
\foreach \x in {0,...,6}
{
	\coordinate (s\x) at (310 + \x*360/7 : 2);
}

\tkzInterLL(s0,s3)(s1,s5) \tkzGetPoint{t4}
\tkzInterLL(s1,s4)(s2,s6) \tkzGetPoint{t5}
\tkzInterLL(s2,s5)(s3,s0) \tkzGetPoint{t6}
\tkzInterLL(s3,s6)(s4,s1) \tkzGetPoint{t0}
\tkzInterLL(s4,s0)(s5,s2) \tkzGetPoint{t1}
\tkzInterLL(s5,s1)(s6,s3) \tkzGetPoint{t2}
\tkzInterLL(s6,s2)(s0,s4) \tkzGetPoint{t3}

 \draw (s0) -- (s3);
 \draw (s1) -- (s4);
 \draw (s2) -- (s5);
 \draw (s3) -- (s6);
 \draw (s4) -- (s0);
 \draw (s5) -- (s1);
 \draw (s6) -- (s2);
 \draw (s0) -- (s2);
 \draw (s1) -- (s3);
 \draw (s2) -- (s4);
 \draw (s3) -- (s5);
 \draw (s4) -- (s6);
 \draw (s5) -- (s0);
 \draw (s6) -- (s1);

\draw[ForestGreen] (s0) node {};
\draw[ForestGreen] (s1) node {};
\draw[NavyBlue] (s2) node {};
\draw[NavyBlue] (s3) node {};
\draw[WildStrawberry] (s4) node {};
\draw[WildStrawberry] (s5) node {};
\draw (s6) node {};

\tkzInterLL(s0,t0)(t6,t1) \tkzGetPoint{r0}
\tkzInterLL(s1,t1)(t0,t2) \tkzGetPoint{r1}
\tkzInterLL(s2,t2)(t1,t3) \tkzGetPoint{r2}
\tkzInterLL(s3,t3)(t2,t4) \tkzGetPoint{r3}
\tkzInterLL(s4,t4)(t3,t5) \tkzGetPoint{r4}
\tkzInterLL(s5,t5)(t4,t6) \tkzGetPoint{r5}
\tkzInterLL(s6,t6)(t5,t0) \tkzGetPoint{r6}

\draw[ForestGreen] (t0) node {};
\draw[NavyBlue] (t1) node {};
\draw[NavyBlue] (t2) node {};
\draw[WildStrawberry] (t3) node {};
\draw[WildStrawberry] (t4) node {};
\draw (t5) node {};
\draw[ForestGreen] (t6) node {};

 \draw (r0) -- (r3);
 \draw (r1) -- (r4);
 \draw (r2) -- (r5);
 \draw (r3) -- (r6);
 \draw (r4) -- (r0);
 \draw (r5) -- (r1);
 \draw (r6) -- (r2);

\draw[WildStrawberry] (r0) node {};
\draw (r1) node {};
\draw[ForestGreen] (r2) node {};
\draw[ForestGreen] (r3) node {};
\draw[NavyBlue] (r4) node {};
\draw[NavyBlue] (r5) node {};
\draw[WildStrawberry] (r6) node {};

\draw[black]
{
	(s1) circle (4pt)
	(s3) circle (4pt)
	(s4) circle (4pt)
	(s6) circle (4pt)
	(t1) circle (4pt)
	(t4) circle (4pt)
	(t6) circle (4pt)
	(r0) circle (4pt)
	(r2) circle (4pt)
	(r4) circle (4pt)
};
\coordinate (tp1) at (-1.9,-1.9);
\end{scope}
\begin{scope}[shift={(8,0)}]
\foreach \x in {0,...,6}
{
	\coordinate (s\x) at (310 + \x*360/7 : 2);
}

\tkzInterLL(s0,s3)(s1,s5) \tkzGetPoint{t4}
\tkzInterLL(s1,s4)(s2,s6) \tkzGetPoint{t5}
\tkzInterLL(s2,s5)(s3,s0) \tkzGetPoint{t6}
\tkzInterLL(s3,s6)(s4,s1) \tkzGetPoint{t0}
\tkzInterLL(s4,s0)(s5,s2) \tkzGetPoint{t1}
\tkzInterLL(s5,s1)(s6,s3) \tkzGetPoint{t2}
\tkzInterLL(s6,s2)(s0,s4) \tkzGetPoint{t3}

 \draw (s0) -- (s3);
 \draw (s1) -- (s4);
 \draw (s2) -- (s5);
 \draw (s3) -- (s6);
 \draw (s4) -- (s0);
 \draw (s5) -- (s1);
 \draw (s6) -- (s2);
 \draw (s0) -- (s2);
 \draw (s1) -- (s3);
 \draw (s2) -- (s4);
 \draw (s3) -- (s5);
 \draw (s4) -- (s6);
 \draw (s5) -- (s0);
 \draw (s6) -- (s1);

\draw[ForestGreen] (s0) node {};
\draw[ForestGreen] (s1) node {};
\draw[NavyBlue] (s2) node {};
\draw[NavyBlue] (s3) node {};
\draw[WildStrawberry] (s4) node {};
\draw[WildStrawberry] (s5) node {};
\draw (s6) node {};

\tkzInterLL(s0,t0)(t6,t1) \tkzGetPoint{r0}
\tkzInterLL(s1,t1)(t0,t2) \tkzGetPoint{r1}
\tkzInterLL(s2,t2)(t1,t3) \tkzGetPoint{r2}
\tkzInterLL(s3,t3)(t2,t4) \tkzGetPoint{r3}
\tkzInterLL(s4,t4)(t3,t5) \tkzGetPoint{r4}
\tkzInterLL(s5,t5)(t4,t6) \tkzGetPoint{r5}
\tkzInterLL(s6,t6)(t5,t0) \tkzGetPoint{r6}

\draw[ForestGreen] (t0) node {};
\draw[NavyBlue] (t1) node {};
\draw[NavyBlue] (t2) node {};
\draw[WildStrawberry] (t3) node {};
\draw[WildStrawberry] (t4) node {};
\draw (t5) node {};
\draw[ForestGreen] (t6) node {};

 \draw (r0) -- (r3);
 \draw (r1) -- (r4);
 \draw (r2) -- (r5);
 \draw (r3) -- (r6);
 \draw (r4) -- (r0);
 \draw (r5) -- (r1);
 \draw (r6) -- (r2);

\draw[WildStrawberry] (r0) node {};
\draw (r1) node {};
\draw[ForestGreen] (r2) node {};
\draw[ForestGreen] (r3) node {};
\draw[NavyBlue] (r4) node {};
\draw[NavyBlue] (r5) node {};
\draw[WildStrawberry] (r6) node {};

\draw[black]
{
	(s1) circle (4pt)
	(s3) circle (4pt)
	(s5) circle (4pt)
	(s6) circle (4pt)
	(t1) circle (4pt)
	(t3) circle (4pt)
	(t6) circle (4pt)
	(r6) circle (4pt)
	(r2) circle (4pt)
	(r4) circle (4pt)
};
\end{scope}
\begin{scope}[shift={(12,0)}]
\foreach \x in {0,...,6}
{
	\coordinate (s\x) at (310 + \x*360/7 : 2);
}
\tkzInterLL(s0,s3)(s1,s5) \tkzGetPoint{t4}
\tkzInterLL(s1,s4)(s2,s6) \tkzGetPoint{t5}
\tkzInterLL(s2,s5)(s3,s0) \tkzGetPoint{t6}
\tkzInterLL(s3,s6)(s4,s1) \tkzGetPoint{t0}
\tkzInterLL(s4,s0)(s5,s2) \tkzGetPoint{t1}
\tkzInterLL(s5,s1)(s6,s3) \tkzGetPoint{t2}
\tkzInterLL(s6,s2)(s0,s4) \tkzGetPoint{t3}

 \draw (s0) -- (s3);
 \draw (s1) -- (s4);
 \draw (s2) -- (s5);
 \draw (s3) -- (s6);
 \draw (s4) -- (s0);
 \draw (s5) -- (s1);
 \draw (s6) -- (s2);
 \draw (s0) -- (s2);
 \draw (s1) -- (s3);
 \draw (s2) -- (s4);
 \draw (s3) -- (s5);
 \draw (s4) -- (s6);
 \draw (s5) -- (s0);
 \draw (s6) -- (s1);

\draw[ForestGreen] (s0) node {};
\draw[ForestGreen] (s1) node {};
\draw[NavyBlue] (s2) node {};
\draw[NavyBlue] (s3) node {};
\draw[WildStrawberry] (s4) node {};
\draw[WildStrawberry] (s5) node {};
\draw (s6) node {};

\tkzInterLL(s0,t0)(t6,t1) \tkzGetPoint{r0}
\tkzInterLL(s1,t1)(t0,t2) \tkzGetPoint{r1}
\tkzInterLL(s2,t2)(t1,t3) \tkzGetPoint{r2}
\tkzInterLL(s3,t3)(t2,t4) \tkzGetPoint{r3}
\tkzInterLL(s4,t4)(t3,t5) \tkzGetPoint{r4}
\tkzInterLL(s5,t5)(t4,t6) \tkzGetPoint{r5}
\tkzInterLL(s6,t6)(t5,t0) \tkzGetPoint{r6}

\draw[ForestGreen] (t0) node {};
\draw[NavyBlue] (t1) node {};
\draw[NavyBlue] (t2) node {};
\draw[WildStrawberry] (t3) node {};
\draw[WildStrawberry] (t4) node {};
\draw (t5) node {};
\draw[ForestGreen] (t6) node {};

 \draw (r0) -- (r3);
 \draw (r1) -- (r4);
 \draw (r2) -- (r5);
 \draw (r3) -- (r6);
 \draw (r4) -- (r0);
 \draw (r5) -- (r1);
 \draw (r6) -- (r2);

\draw[WildStrawberry] (r0) node {};
\draw (r1) node {};
\draw[ForestGreen] (r2) node {};
\draw[ForestGreen] (r3) node {};
\draw[NavyBlue] (r4) node {};
\draw[NavyBlue] (r5) node {};
\draw[WildStrawberry] (r6) node {};

\draw[black]
{
	(s1) circle (4pt)
	(s2) circle (4pt)
	(s5) circle (4pt)
	(s6) circle (4pt)
	(t2) circle (4pt)
	(t3) circle (4pt)
	(t6) circle (4pt)
	(r6) circle (4pt)
	(r2) circle (4pt)
	(r5) circle (4pt)
};
\end{scope}
\begin{scope}[shift = {(0,-4.5)}]
\foreach \x in {0,...,6}
{
	\coordinate (s\x) at (310 + \x*360/7 : 2);
}

\tkzInterLL(s0,s3)(s1,s5) \tkzGetPoint{t4}
\tkzInterLL(s1,s4)(s2,s6) \tkzGetPoint{t5}
\tkzInterLL(s2,s5)(s3,s0) \tkzGetPoint{t6}
\tkzInterLL(s3,s6)(s4,s1) \tkzGetPoint{t0}
\tkzInterLL(s4,s0)(s5,s2) \tkzGetPoint{t1}
\tkzInterLL(s5,s1)(s6,s3) \tkzGetPoint{t2}
\tkzInterLL(s6,s2)(s0,s4) \tkzGetPoint{t3}

 \draw (s0) -- (s3);
 \draw (s1) -- (s4);
 \draw (s2) -- (s5);
 \draw (s3) -- (s6);
 \draw (s4) -- (s0);
 \draw (s5) -- (s1);
 \draw (s6) -- (s2);
 \draw (s0) -- (s2);
 \draw (s1) -- (s3);
 \draw (s2) -- (s4);
 \draw (s3) -- (s5);
 \draw (s4) -- (s6);
 \draw (s5) -- (s0);
 \draw (s6) -- (s1);

\draw[ForestGreen] (s0) node {};
\draw[ForestGreen] (s1) node {};
\draw[NavyBlue] (s2) node {};
\draw[NavyBlue] (s3) node {};
\draw[WildStrawberry] (s4) node {};
\draw[WildStrawberry] (s5) node {};
\draw (s6) node {};

\tkzInterLL(s0,t0)(t6,t1) \tkzGetPoint{r0}
\tkzInterLL(s1,t1)(t0,t2) \tkzGetPoint{r1}
\tkzInterLL(s2,t2)(t1,t3) \tkzGetPoint{r2}
\tkzInterLL(s3,t3)(t2,t4) \tkzGetPoint{r3}
\tkzInterLL(s4,t4)(t3,t5) \tkzGetPoint{r4}
\tkzInterLL(s5,t5)(t4,t6) \tkzGetPoint{r5}
\tkzInterLL(s6,t6)(t5,t0) \tkzGetPoint{r6}

\draw[ForestGreen] (t0) node {};
\draw[NavyBlue] (t1) node {};
\draw[NavyBlue] (t2) node {};
\draw[WildStrawberry] (t3) node {};
\draw[WildStrawberry] (t4) node {};
\draw (t5) node {};
\draw[ForestGreen] (t6) node {};

 \draw (r0) -- (r3);
 \draw (r1) -- (r4);
 \draw (r2) -- (r5);
 \draw (r3) -- (r6);
 \draw (r4) -- (r0);
 \draw (r5) -- (r1);
 \draw (r6) -- (r2);

\draw[WildStrawberry] (r0) node {};
\draw (r1) node {};
\draw[ForestGreen] (r2) node {};
\draw[ForestGreen] (r3) node {};
\draw[NavyBlue] (r4) node {};
\draw[NavyBlue] (r5) node {};
\draw[WildStrawberry] (r6) node {};

\draw[black]
{
	(s1) circle (4pt)
	(s2) circle (4pt)
	(s4) circle (4pt)
	(t2) circle (4pt)
	(t4) circle (4pt)
	(t6) circle (4pt)
	(t5) circle (4pt)
	(r0) circle (4pt)
	(r2) circle (4pt)
	(r5) circle (4pt)
	(r1) circle (4pt)

};
\end{scope}
\begin{scope}[shift = {(4,-4.5)}]
\foreach \x in {0,...,6}
{
	\coordinate (s\x) at (310 + \x*360/7 : 2);
}

\tkzInterLL(s0,s3)(s1,s5) \tkzGetPoint{t4}
\tkzInterLL(s1,s4)(s2,s6) \tkzGetPoint{t5}
\tkzInterLL(s2,s5)(s3,s0) \tkzGetPoint{t6}
\tkzInterLL(s3,s6)(s4,s1) \tkzGetPoint{t0}
\tkzInterLL(s4,s0)(s5,s2) \tkzGetPoint{t1}
\tkzInterLL(s5,s1)(s6,s3) \tkzGetPoint{t2}
\tkzInterLL(s6,s2)(s0,s4) \tkzGetPoint{t3}

 \draw (s0) -- (s3);
 \draw (s1) -- (s4);
 \draw (s2) -- (s5);
 \draw (s3) -- (s6);
 \draw (s4) -- (s0);
 \draw (s5) -- (s1);
 \draw (s6) -- (s2);
 \draw (s0) -- (s2);
 \draw (s1) -- (s3);
 \draw (s2) -- (s4);
 \draw (s3) -- (s5);
 \draw (s4) -- (s6);
 \draw (s5) -- (s0);
 \draw (s6) -- (s1);

\draw[ForestGreen] (s0) node {};
\draw[ForestGreen] (s1) node {};
\draw[NavyBlue] (s2) node {};
\draw[NavyBlue] (s3) node {};
\draw[WildStrawberry] (s4) node {};
\draw[WildStrawberry] (s5) node {};
\draw (s6) node {};

\tkzInterLL(s0,t0)(t6,t1) \tkzGetPoint{r0}
\tkzInterLL(s1,t1)(t0,t2) \tkzGetPoint{r1}
\tkzInterLL(s2,t2)(t1,t3) \tkzGetPoint{r2}
\tkzInterLL(s3,t3)(t2,t4) \tkzGetPoint{r3}
\tkzInterLL(s4,t4)(t3,t5) \tkzGetPoint{r4}
\tkzInterLL(s5,t5)(t4,t6) \tkzGetPoint{r5}
\tkzInterLL(s6,t6)(t5,t0) \tkzGetPoint{r6}

\draw[ForestGreen] (t0) node {};
\draw[NavyBlue] (t1) node {};
\draw[NavyBlue] (t2) node {};
\draw[WildStrawberry] (t3) node {};
\draw[WildStrawberry] (t4) node {};
\draw (t5) node {};
\draw[ForestGreen] (t6) node {};

 \draw (r0) -- (r3);
 \draw (r1) -- (r4);
 \draw (r2) -- (r5);
 \draw (r3) -- (r6);
 \draw (r4) -- (r0);
 \draw (r5) -- (r1);
 \draw (r6) -- (r2);

\draw[WildStrawberry] (r0) node {};
\draw (r1) node {};
\draw[ForestGreen] (r2) node {};
\draw[ForestGreen] (r3) node {};
\draw[NavyBlue] (r4) node {};
\draw[NavyBlue] (r5) node {};
\draw[WildStrawberry] (r6) node {};

\draw[black]
{
	(s1) circle (4pt)
	(s3) circle (4pt)
	(s4) circle (4pt)
	(t1) circle (4pt)
	(t4) circle (4pt)
	(t6) circle (4pt)
	(t5) circle (4pt)
	(r0) circle (4pt)
	(r2) circle (4pt)
	(r4) circle (4pt)
	(r1) circle (4pt)
};
\end{scope}
\begin{scope}[shift = {(8,-4.5)}]
\foreach \x in {0,...,6}
{
	\coordinate (s\x) at (310 + \x*360/7 : 2);
}

\tkzInterLL(s0,s3)(s1,s5) \tkzGetPoint{t4}
\tkzInterLL(s1,s4)(s2,s6) \tkzGetPoint{t5}
\tkzInterLL(s2,s5)(s3,s0) \tkzGetPoint{t6}
\tkzInterLL(s3,s6)(s4,s1) \tkzGetPoint{t0}
\tkzInterLL(s4,s0)(s5,s2) \tkzGetPoint{t1}
\tkzInterLL(s5,s1)(s6,s3) \tkzGetPoint{t2}
\tkzInterLL(s6,s2)(s0,s4) \tkzGetPoint{t3}

 \draw (s0) -- (s3);
 \draw (s1) -- (s4);
 \draw (s2) -- (s5);
 \draw (s3) -- (s6);
 \draw (s4) -- (s0);
 \draw (s5) -- (s1);
 \draw (s6) -- (s2);
 \draw (s0) -- (s2);
 \draw (s1) -- (s3);
 \draw (s2) -- (s4);
 \draw (s3) -- (s5);
 \draw (s4) -- (s6);
 \draw (s5) -- (s0);
 \draw (s6) -- (s1);

\draw[ForestGreen] (s0) node {};
\draw[ForestGreen] (s1) node {};
\draw[NavyBlue] (s2) node {};
\draw[NavyBlue] (s3) node {};
\draw[WildStrawberry] (s4) node {};
\draw[WildStrawberry] (s5) node {};
\draw (s6) node {};

\tkzInterLL(s0,t0)(t6,t1) \tkzGetPoint{r0}
\tkzInterLL(s1,t1)(t0,t2) \tkzGetPoint{r1}
\tkzInterLL(s2,t2)(t1,t3) \tkzGetPoint{r2}
\tkzInterLL(s3,t3)(t2,t4) \tkzGetPoint{r3}
\tkzInterLL(s4,t4)(t3,t5) \tkzGetPoint{r4}
\tkzInterLL(s5,t5)(t4,t6) \tkzGetPoint{r5}
\tkzInterLL(s6,t6)(t5,t0) \tkzGetPoint{r6}

\draw[ForestGreen] (t0) node {};
\draw[NavyBlue] (t1) node {};
\draw[NavyBlue] (t2) node {};
\draw[WildStrawberry] (t3) node {};
\draw[WildStrawberry] (t4) node {};
\draw (t5) node {};
\draw[ForestGreen] (t6) node {};

 \draw (r0) -- (r3);
 \draw (r1) -- (r4);
 \draw (r2) -- (r5);
 \draw (r3) -- (r6);
 \draw (r4) -- (r0);
 \draw (r5) -- (r1);
 \draw (r6) -- (r2);

\draw[WildStrawberry] (r0) node {};
\draw (r1) node {};
\draw[ForestGreen] (r2) node {};
\draw[ForestGreen] (r3) node {};
\draw[NavyBlue] (r4) node {};
\draw[NavyBlue] (r5) node {};
\draw[WildStrawberry] (r6) node {};

\draw[black]
{
	(s0) circle (4pt)
	(s3) circle (4pt)
	(s4) circle (4pt)
	(s6) circle (4pt)
	(t1) circle (4pt)
	(t4) circle (4pt)
	(t0) circle (4pt)
	(r0) circle (4pt)
	(r3) circle (4pt)
	(r4) circle (4pt)
};
\end{scope}
\begin{scope}[shift = {(12,-4.5)}]
\foreach \x in {0,...,6}
{
	\coordinate (s\x) at (310 + \x*360/7 : 2);
}

\tkzInterLL(s0,s3)(s1,s5) \tkzGetPoint{t4}
\tkzInterLL(s1,s4)(s2,s6) \tkzGetPoint{t5}
\tkzInterLL(s2,s5)(s3,s0) \tkzGetPoint{t6}
\tkzInterLL(s3,s6)(s4,s1) \tkzGetPoint{t0}
\tkzInterLL(s4,s0)(s5,s2) \tkzGetPoint{t1}
\tkzInterLL(s5,s1)(s6,s3) \tkzGetPoint{t2}
\tkzInterLL(s6,s2)(s0,s4) \tkzGetPoint{t3}

 \draw (s0) -- (s3);
 \draw (s1) -- (s4);
 \draw (s2) -- (s5);
 \draw (s3) -- (s6);
 \draw (s4) -- (s0);
 \draw (s5) -- (s1);
 \draw (s6) -- (s2);
 \draw (s0) -- (s2);
 \draw (s1) -- (s3);
 \draw (s2) -- (s4);
 \draw (s3) -- (s5);
 \draw (s4) -- (s6);
 \draw (s5) -- (s0);
 \draw (s6) -- (s1);

\draw[ForestGreen] (s0) node {};
\draw[ForestGreen] (s1) node {};
\draw[NavyBlue] (s2) node {};
\draw[NavyBlue] (s3) node {};
\draw[WildStrawberry] (s4) node {};
\draw[WildStrawberry] (s5) node {};
\draw (s6) node {};

\tkzInterLL(s0,t0)(t6,t1) \tkzGetPoint{r0}
\tkzInterLL(s1,t1)(t0,t2) \tkzGetPoint{r1}
\tkzInterLL(s2,t2)(t1,t3) \tkzGetPoint{r2}
\tkzInterLL(s3,t3)(t2,t4) \tkzGetPoint{r3}
\tkzInterLL(s4,t4)(t3,t5) \tkzGetPoint{r4}
\tkzInterLL(s5,t5)(t4,t6) \tkzGetPoint{r5}
\tkzInterLL(s6,t6)(t5,t0) \tkzGetPoint{r6}

\draw[ForestGreen] (t0) node {};
\draw[NavyBlue] (t1) node {};
\draw[NavyBlue] (t2) node {};
\draw[WildStrawberry] (t3) node {};
\draw[WildStrawberry] (t4) node {};
\draw (t5) node {};
\draw[ForestGreen] (t6) node {};

 \draw (r0) -- (r3);
 \draw (r1) -- (r4);
 \draw (r2) -- (r5);
 \draw (r3) -- (r6);
 \draw (r4) -- (r0);
 \draw (r5) -- (r1);
 \draw (r6) -- (r2);

\draw[WildStrawberry] (r0) node {};
\draw (r1) node {};
\draw[ForestGreen] (r2) node {};
\draw[ForestGreen] (r3) node {};
\draw[NavyBlue] (r4) node {};
\draw[NavyBlue] (r5) node {};
\draw[WildStrawberry] (r6) node {};

\draw[black]
{
	(s0) circle (4pt)
	(s3) circle (4pt)
	(s5) circle (4pt)
	(s6) circle (4pt)
	(t1) circle (4pt)
	(t3) circle (4pt)
	(t0) circle (4pt)
	(r6) circle (4pt)
	(r3) circle (4pt)
	(r4) circle (4pt)
};
\end{scope}
\end{tikzpicture}
\caption{The legal system of Brinkmann graph has $4$ moves each corresponding to vertices in one of $4$ colors in the figure. An orbit of a state under these moves consists of $16$ states. In the figure we illustrate $8$ states of a legal orbit, the $8$ other states are ``complementary'' to the illustrated states, i.e. they are obtained from the illustrated states by exchanging circled and uncircled vertices. Thus to ensure that the orbit is legal it suffices to check only $8$ illustrated states.}
\label{fig:Brinkmann}
\end{figure}
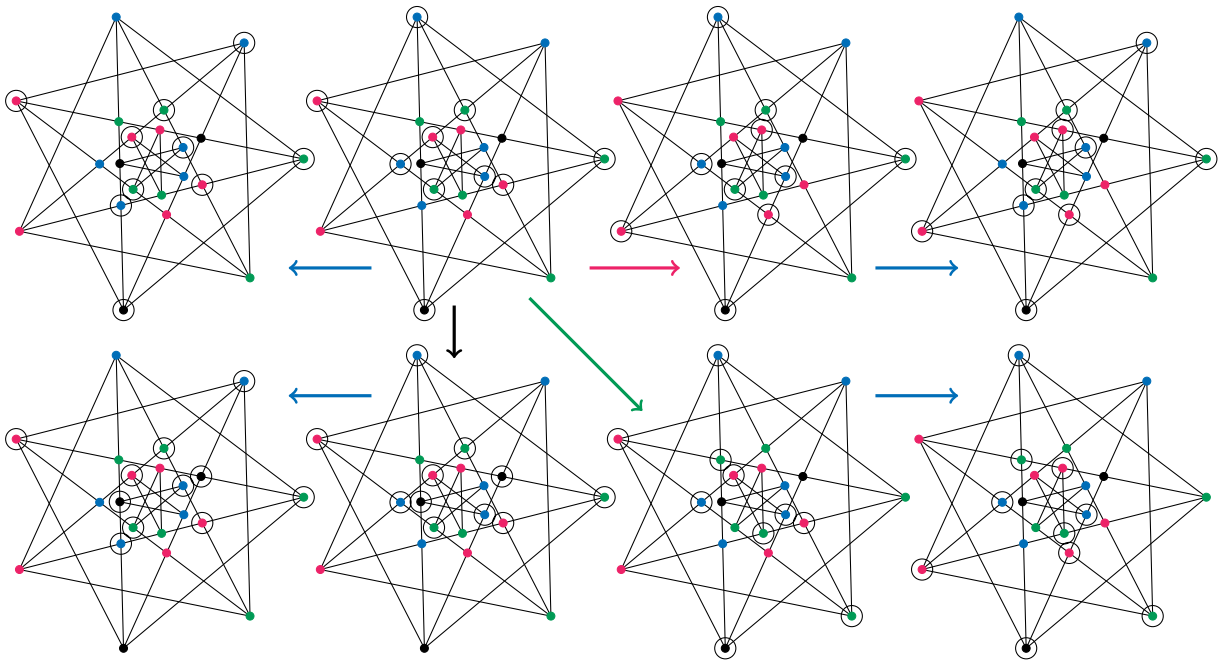

\subsection{Blowup of a cube}

Let $\bar \Gamma$ be the 1-skeleton of a $d$-cube.
We will produce a graph $\Gamma$ and a map $\rho: \Gamma\rightarrow \bar \Gamma$ with the following properties:
There is a number $n$ and we will later conveniently assume that $n$ is a large prime.
For each $\bar v\in \bar \Gamma^0$, its preimage $\rho^{-1}(\bar v)$ consists
of a set of vertices $v_i : 1\leq i \leq n$.

We distinguish one of the three  parallelism classes of edges of $\bar \Gamma$.

Let $\bar e$ be a distinguished edge of $\bar \Gamma$ whose endpoints are $\bar u, \bar v$.
Its preimage in $\Gamma$ consists of a set of $2n-1$ edges each of which maps to $\bar e$,
such that their union with $u_1,\ldots, u_n, v_1,\ldots, v_n$ is a tree.
(In practice $\rho^{-1}(e)$ is obtained from a $2n$-cycle by removing an edge.)

Let $\bar e$ be a non-distinguished edge of $\bar \Gamma$ whose endpoints are $\bar u,\bar v$.
Its preimage consists of $n$ edges mapping to $\bar e$, such that the correspondence between their endpoints
yields a bijection $\{u_1,\ldots, u_n\} \leftrightarrow \{v_1,\ldots, v_n\}$.

We now explain that for strategic choices above we have  $\girth(\Gamma)\geq 6$.
First observe that $\Gamma$ is bipartite, and also simplicial by construction, and so it suffices to exclude
$4$-cycles.
A cycle in $\Gamma$ projects to a combinatorial path in $\bar \Gamma$.
Observe that a $4$-cycle cannot project to a single edge, since the preimage of a distinguished edge is a tree,
and the preimage of a non-distinguished edge is a disjoint union of edges.
Likewise, the $4$-cycle cannot project to the union of two edges, since at least one of these is non-distinguished,
and so consecutive edges of the $4$-cycle mapping to this non-distinguished edge would not share an endpoint.
The remaining possibility is that the $4$-cycle maps to a $4$-cycle of $\bar \Gamma$.
It thus suffices to exclude 4-cycles over the 4-cycles of $\bar \Gamma$.

When $n$ is large, one can choose the trees and the bijections at random for each $e$ and the conclusion almost always holds. Moreover, $\girth(\Gamma)$ can be ensured to be arbitrarily large for large $n$ in this case.

Alternately, for the distinguished edge $e$ with vertices $\bar u,\bar v$ we use  a ``linear tree'' $u_1 - v_1 - u_2 - v_2 - \cdots - v_{n-1} - u_n - v_n$.
And we associate a natural number $ k_e < \frac{n}{2}$ to each non-distinguished edge $e$,
and use the bijection $u_i \mapsto v_{i+k_e}$ with subscripts taken modulo $n$.
And we then require that for edges $\bar e, \bar e'$ that are parallel in a square of $\bar \Gamma$,
the associated numbers $k_e, k_{e'}$ satisfy $|k_e- k_{e'}|\geq 3$.

We note that when $d>3$, the graph $\Gamma$ has $\curvature(\Gamma)>0$.
However, when $d=3$, we have $\curvature(\Gamma)=0$.
Indeed,  $\Gamma$ then has $8n$ vertices, and $8n+4(2n-1)$ edges.
Thus $\curvature(\Gamma)= 1- \frac{8n}{2}+\frac{8n+4(2n-1)}{4} = 1-4n + 2n+ (2n-1) =0$.

A legal system for $\Gamma$ is the ``preimage'' of a legal system for $\bar \Gamma$ as described in
Example~\ref{ex:cube}: \begin{com} this works for arbitrary cubes - and even cubes minus some edges...\end{com}
Its moves correspond to a 2-coloring of $\Gamma$ and its states correspond to preimages of states
in the legal system for $\bar \Gamma$.
To see that each such state $S$ is connected, let $\bar S$ be the corresponding state in $\bar \Gamma$.
 For each vertex $v_i$ in $S$ there is a path in $\bar S$ from $\bar v$ to a distinguished edge $\bar e$.
 This path lifts to a path in $S$ that ends at the preimage of $\bar e$ which is a tree.

\begin{rem}
It appears that the key point to generalizing this ``blow-up'' construction of a graph $\bar \Gamma$ with a legal system, is that there is a distinguished edge in each state.
\end{rem}

\subsection{L\"{o}bell graphs and right-angled 3-dimensional hyperbolic reflection groups}

\begin{defn}[L\"{o}bell graphs]
The \emph{$n$-antiprism} is the polyhedron whose boundary consists of two disjoint $n$-gon faces joined by an annulus that is subdivided into $2n$ triangles, so that there is one $n$-gon and three triangles at each vertex.
The \emph{dual L\"{o}bell graph} of degree $n\geq 4$ is the 1-skeleton of the polyhedron obtained from an $n$-antiprism by centrally subdividing each $n$-gon into $n$~triangles. For example, for $n=5$ we obtain an icosahedron. The L\"{o}bell graphs themselves are formed from pentagons and $n$-gons, and early examples beyond the icosahedron occurred in L\"{o}bells construction of certain hyperbolic 3-manifolds. They were further studied by Vesnin in \cite{Vesnin87}.
\end{defn}

\begin{lem}\label{lem:LobellGraphsWin}
Every dual L\"{o}bell graph of degree $n \geq 5$ has a legal system.
\end{lem}
\begin{proof}
If $n=5$ the dual L\"obell graph is the icosahedron and has a legal system described in Section~\ref{sub:icosahedron}.
Suppose $n \geq 6$.
Consider the following coloring of the dual L\"obell graph with $n+1$ colors having two vertices of each color:
\begin{itemize}
\item the vertices $\{v_*,w_*\}$ that are dual to the  $n$-gonal faces have color $c_*$,
\item let $v_1,\dots, v_n$ and $w_1,\dots , w_n$ denote consecutively the vertices adjacent to $v_*$ and $w_*$. And ``shift'' the subscripts  so that $w_i$ is adjacent to $v_{i-2}$ and $v_{i-1}$ for each $i$ (mod $n$). Each pair $\{v_i,w_i\}$ has color  $c_i$.
See Figure~\ref{fig:lobell coloring}.
\end{itemize}
There is a move corresponding to $c_*$. There is a move corresponding to $c_i$ for each $i \notin \{1, \lceil \frac{n}{2} \rceil\}$.
There is also a move corresponding to the four vertices colored by $c_1$ or $c_{\lceil\frac{n}{2}\rceil}$.
This last move is valid as $c_1$-vertices are not adjacent to $c_{\lceil\frac{n}{2}\rceil}$-vertices  for $n\geq 6$ because of our subscript shift.
Let $S$ be a state containing exactly one $c_*$-vertex, exactly one $c_i$-vertex  for each $i\notin\{1,
{\lceil\frac{n}{2}\rceil}\}$, and also containing either $v_{1}$ and $w_{\lceil\frac{n}{2}\rceil}$ or $v_{\lceil\frac{n}{2}\rceil}$
and $w_{1}$.
We now verify that $S$ and the complement $\bar S$ of $S$ are connected.
Let  $V=\{v_1,\dots, v_n\}$ and $W=\{w_1,\dots , w_n\}$.
It suffices to show that each vertex $v\in V\cap S$ lies in the same connected component as some $w\in W\cap S$, and that each $w\in W\cap S$ lies in the same connected component as  some $v\in V\cap S$.
Indeed, since one $c_*$-vertex is in $S$,
either all the vertices in $V\cap S$ or all the vertices in $W\cap S$ are adjacent to it, and it follows that $S$ is connected.
Consider any maximal string $\underline v = \{v_i,v_{i+1}, \dots, v_k\}\subset V$ (with indices mod $n$) that is in $S$. We have $\underline v\subsetneq V$, since only one of $v_1,v_{\lceil\frac n 2 \rceil}$ is in $S$. Note that $w_{k+1}$ is in $S$, since $v_{k+1}\in \bar S$ by maximality of $\underline v$. Since $w_{k+1}$ is adjacent to $v_k$ we conclude that  each vertex of $\underline v$ lies in the connected component of some vertex in $W\cap S$. Similarly consider a maximal string $\underline w = \{w_i,\dots w_k\}\subset W$ that is in $S$ and now $v_{i-1}$ is in $S$ by maximality of $\underline w$. Since $v_{i-1}$ is adjacent to $w_{i}$, we conclude that each vertex of $W\cap S$ lies in a connected component of some  vertex in $V\cap S$. We verify that $\bar S$ is connected in the exact same way.
Let $S_0=\{v_*, w_1, v_2, v_3,\ldots, v_n\}$.
Then each element of the orbit of $S_0$ using the above moves is legal and has legal complement.
\end{proof}

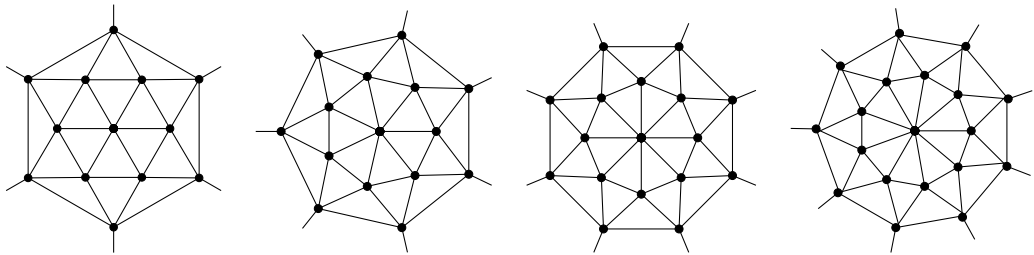
\begin{figure}
\begin{tikzpicture}[scale=0.75]
\tikzstyle{every node}=[circle, draw, fill,
                        inner sep=1pt, minimum width = 2pt]
\foreach \s in {0,...,5}
{
	\draw (0:0) node {};
	\draw (0:0) to (\s*60:1) node {} to (\s*60+60:1);
	\draw (\s*60-30:2.2) to (\s*60-30:1.75) node {} to (\s*60+30:1.75);
	\draw (\s*60-30:1.75) to (\s*60:1) to (\s*60+30:1.75);
}
\end{tikzpicture}\;\;\;
\begin{tikzpicture}[scale=0.75]
\tikzstyle{every node}=[circle, draw, fill,
                        inner sep=1pt, minimum width = 3pt]
\foreach \s in {0,...,6}
{
	\draw (0:0) node {};
	\draw (0:0) to (\s*360/7:1) node {} to (\s*360/7+360/7:1);
	\draw (\s*360/7-180/7:2.2) to (\s*360/7-180/7:1.75) node {} to (\s*360/7+180/7:1.75);
	\draw (\s*360/7-180/7:1.75) to (\s*360/7:1) to (\s*360/7+180/7:1.75);
}
\end{tikzpicture}\;\;\;
\begin{tikzpicture}[scale=0.75]
\tikzstyle{every node}=[circle, draw, fill,
                        inner sep=1pt, minimum width = 3pt]
\foreach \s in {0,...,7}
{
	\draw (0:0) node {};
	\draw (0:0) to (\s*360/8:1) node {} to (\s*360/8+360/8:1);
	\draw (\s*360/8-180/8:2.2) to (\s*360/8-180/8:1.75) node {} to (\s*360/8+180/8:1.75);
	\draw (\s*360/8-180/8:1.75) to (\s*360/8:1) to (\s*360/8+180/8:1.75);
}
\end{tikzpicture}\;\;\;
\begin{tikzpicture}[scale=0.75]
\tikzstyle{every node}=[circle, draw, fill,
                        inner sep=1pt, minimum width = 3pt]
\foreach \s in {0,...,8}
{
	\draw (0:0) node {};
	\draw (0:0) to (\s*360/9:1) node {} to (\s*360/9+360/9:1);
	\draw (\s*360/9-190/9:2.2) to (\s*360/9-190/9:1.75) node {} to (\s*360/9+190/9:1.75);
	\draw (\s*360/9-190/9:1.75) to (\s*360/9:1) to (\s*360/9+190/9:1.75);
}
\end{tikzpicture}
\caption{L\"{o}bell graphs of degree $6$,$7$,$8$ and $9$. In each graph there is one vertex ``at infinity'' not illustrated in the figure.}\label{fig:lobell}\end{figure}

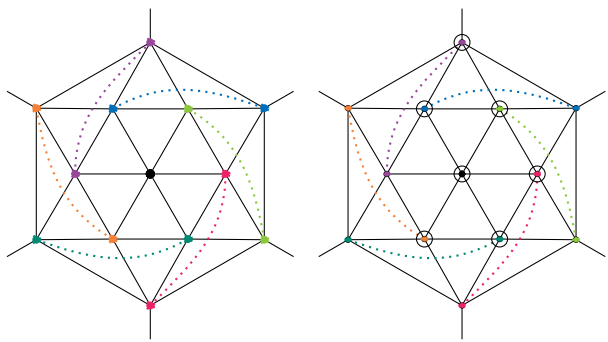
\begin{figure}
\[
\begin{tikzpicture}
\tikzstyle{every node}=[circle, draw, fill,
                        inner sep=1pt, minimum width = 3pt]
\foreach \s in {0,...,5}
{
	\draw (0:0) node {};
	\draw (0:0) to (\s*60:1) to (\s*60+60:1);
	\draw (\s*60-30:2.2) to (\s*60-30:1.75)  to (\s*60+30:1.75);
	\draw (\s*60-30:1.75) to (\s*60:1) to (\s*60+30:1.75);
}
\draw[WildStrawberry,dotted, thick] (0:1) node {} to[out=270, in=45] (270:1.75) node {};
\draw[LimeGreen,dotted,thick] (60:1) node {} to[out= -30,in=100] (330:1.75) node {};
\draw[NavyBlue,dotted,thick] (120:1) node {} to[out=30, in=160] (30:1.75) node {};
\draw[Purple,dotted, thick] (180:1) node {} to[out=90, in=225] (90:1.75) node {};
\draw[Orange,dotted,thick] (240:1) node {} to[out = 150, in = 280] (150:1.75) node {};
\draw[PineGreen,dotted,thick] (300:1) node {} to[out = 210, in = -20] (210:1.75) node {};
\end{tikzpicture}
\;\;\;
\begin{tikzpicture}
\tikzstyle{every node}=[circle, draw, fill,
                        inner sep=0.5pt, minimum width = 2pt]
\foreach \s in {0,...,5}
{
	\draw (0:0) node {};
	\draw (0:0) to (\s*60:1) node {} to (\s*60+60:1);
	\draw (\s*60-30:2.2) to (\s*60-30:1.75) node {} to (\s*60+30:1.75);
	\draw (\s*60-30:1.75) to (\s*60:1) to (\s*60+30:1.75);
}

\draw[WildStrawberry,dotted, thick] (0:1) node {} to[out=270, in=45] (270:1.75) node {};
\draw[LimeGreen,dotted,thick] (60:1) node {} to[out= -30,in=100] (330:1.75) node {};
\draw[NavyBlue,dotted,thick] (120:1) node {} to[out=30, in=160] (30:1.75) node {};
\draw[Purple,dotted, thick] (180:1) node {} to[out=90, in=225] (90:1.75) node {};
\draw[Orange,dotted,thick] (240:1) node {} to[out = 150, in = 280] (150:1.75) node {};
\draw[PineGreen,dotted,thick] (300:1) node {} to[out = 210, in = -20] (210:1.75) node {};
	\draw (0:0) circle (3pt);
	\draw (0:1) circle (3pt);
	\draw (60:1) circle (3pt);
	\draw (120:1) circle (3pt);
	\draw (90:1.75) circle (3pt);
	\draw (240:1) circle (3pt);
	\draw (300:1) circle (3pt);
\end{tikzpicture}
\]
\caption{\label{fig:lobell coloring}\label{fig:lobell system}At left is the coloring of the degree $6$ dual L\"{o}bell graph from the proof of Lemma~\ref{lem:LobellGraphsWin}. The vertex ``at infinity'' is paired with the central vertex. At right is an example of a legal state whose orbit is legal.}
\end{figure}


We recall the following result of Pogorelov characterizing 3-dimensional compact right-angled hyperbolic polyhedra  \cite{Pogorelov67}, which we restate in terms of Coxeter groups:

\begin{thm}\label{thm:pogorelov}
For a finite simplicial graph $\Gamma$, the group
$G(\Gamma)$ is a cocompact 3-dimensional right-angled hyperbolic reflection group if and only if
$\Gamma$ embeds in $S^2$ so that:
\begin{enumerate}
\item each region is a triangle;
\item each cycle of length $\leq4$ either bounds a triangle or two triangles meeting at an edge;
\item $\Gamma$ is not a triangle or a 4-clique.
\end{enumerate}
\end{thm}

It follows that each vertex has valence $\geq 5$ since a low valence vertex would lead to a contradiction. For instance, a valence~$4$ vertex would be surrounded by a 4-cycle,
which would then have to bound a pair of triangles meeting an edge, in which case there would be a triangle not bounding a region.

We refer to graphs satisfying the conditions in Theorem~\ref{thm:pogorelov} as \emph{Pogorelov graphs}.
Inoue gave a recursive construction of all (duals of) Pogorelov graphs \cite{Inoue08}.
Expressing his result in terms of duals,  the base case consists of the dual L\"{o}bell graphs.
His structural inductions are of two types:
\begin{enumerate}
\item
Combine two Pogorelov graphs by removing a valence~$n$ vertex from each and then gluing
together along an $n$-antiprism as on the top of Figure~\ref{fig:Inoue}

\item
Expand a vertex of a Pogorelov graph to an edge whose vertices each have valence $\geq 5$ as in the bottom of Figure~\ref{fig:Inoue}.
\end{enumerate}

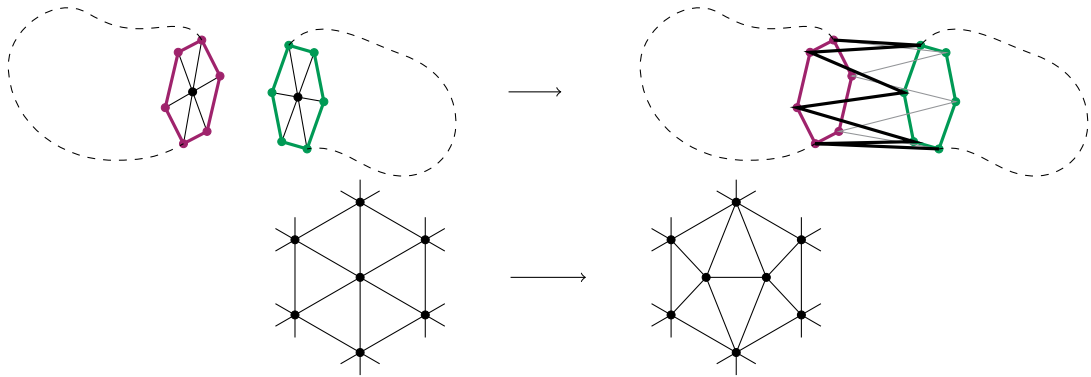
\begin{figure}
\begin{tikzpicture}[scale = 0.7]
\tikzstyle{every node}=[circle, draw, fill,
                        inner sep=0pt, minimum width = 3pt]
\draw[->] (6,0) -- (7,0);
\begin{scope}
\coordinate (o) at (0,0);
\coordinate (a) at (80:1);
\coordinate (d) at (260:1);
\coordinate (b) at (110:0.8);
\coordinate (e) at (290: 0.8);
\coordinate (c) at (210:0.6);
\coordinate (f) at (30:0.6);
\node at (o) {};
\foreach \x in {a,b,c,d,e,f}
{
\draw (o) -- (\x) node[RedViolet] {};
}
\draw[very thick, RedViolet] (a) -- (b) -- (c) -- (d) -- (e) -- (f) -- (a);
\draw[dashed] (a) to[out = 120, in = 0] (-1,1.2) to [out =180, in = -30] (-2,1.5) to [out = 150, in =90]  (-3.5,0.5)
 to [out = 270, in = 180] (-1.5, -1.3) to [out = 0, in =210] (d);
 \end{scope}
 \begin{scope}[shift = {(2,-0.1)}]
 \coordinate (o') at (0,0);
\coordinate (a') at (100:1);
\coordinate (d') at (280:1);
\coordinate (b') at (70:0.9);
\coordinate (e') at (250: 0.9);
\coordinate (c') at (350:0.5);
\coordinate (f') at (170:0.5);
\node at (o') {};
\foreach \x in {a',b',c',d',e',f'}
{
\draw (o') -- (\x) node[ForestGreen] {};
}
\draw[very thick, ForestGreen] (a') -- (b') -- (c') -- (d') -- (e') -- (f') -- (a');
\draw[dashed] (a') to[out =50, in =160 ] (2,0.8) to[out =-20, in = 90] (3,-.5) to[out =270, in = 0] (2,-1.5) to[out =180, in = 10] (d');
 \end{scope}
\begin{scope}[shift = {(12,0)}]
\coordinate (o) at (0,0);
\coordinate (a) at (80:1);
\coordinate (d) at (260:1);
\coordinate (b) at (110:0.8);
\coordinate (e) at (290: 0.8);
\coordinate (c) at (210:0.6);
\coordinate (f) at (30:0.6);
\foreach \x in {a,b,c,d,e,f}
{
\node[RedViolet] at (\x) {};
}
\draw[very thick, RedViolet] (a) -- (b) -- (c) -- (d) -- (e) -- (f) -- (a);
\draw[dashed] (a) to[out = 120, in = 0] (-1,1.2) to [out =180, in = -30] (-2,1.5) to [out = 150, in =90]  (-3.5,0.5)
 to [out = 270, in = 180] (-1.5, -1.3) to [out = 0, in =210] (d);
 \end{scope}
 \begin{scope}[shift = {(14,-0.1)}]
 \coordinate (o') at (0,0);
\coordinate (a') at (100:1);
\coordinate (d') at (280:1);
\coordinate (f') at (70:0.9);
\coordinate (c') at (250: 0.9);
\coordinate (e') at (350:0.5);
\coordinate (b') at (170:0.5);
\foreach \x in {a',b',c',d',e',f'}
{
\node[ForestGreen] at (\x) {};
}
\draw[very thick, ForestGreen] (a') -- (b') -- (c') -- (d') -- (e') -- (f') -- (a');
\draw[dashed] (a') to[out =50, in =160 ] (2,0.8) to[out =-20, in = 90] (3,-.5) to[out =270, in = 0] (2,-1.5) to[out =180, in = 10] (d');
 \end{scope}
 \draw[thin, Gray] (d')-- (e) -- (e') -- (f) -- (f') -- (a);
  \draw[very thick] (a) -- (a') -- (b) -- (b') -- (c) -- (c') -- (d) -- (d');

 \end{tikzpicture}

\begin{tikzpicture}
\tikzstyle{every node}=[circle, draw, fill,
                        inner sep=0pt, minimum width = 3pt]
\begin{scope}
\node at (0,0) {};
\foreach \x in {-30,30,90,150,210,270}
{
\draw (0,0) -- (\x:1) node {} -- (\x+60:1);
\draw[shift = {(\x:1)}] (0,0) -- (\x:0.3);
\draw[shift = {(\x:1)}] (0,0) -- (\x-60:0.3);
\draw[shift = {(\x:1)}] (0,0) -- (\x+60:0.3);
}
\end{scope}
\begin{scope}[shift = {(2,0)}]
\draw [->] (0,0) to (1,0);
\end{scope}
\begin{scope}[shift = {(5,0)}]
\coordinate (o1) at (-.4,0);
\coordinate (o2) at (.4,0);
\draw (o1) node {} -- (o2) node {};
\foreach \x in {150,210}
{
\draw (o1) -- (\x:1) node {} -- (\x+60:1);
\draw[shift = {(\x:1)}] (0,0) -- (\x:0.3);
\draw[shift = {(\x:1)}] (0,0) -- (\x-60:0.3);
\draw[shift = {(\x:1)}] (0,0) -- (\x+60:0.3);
}
\foreach \x in {30, -30}
{
\draw (o2) -- (\x:1) node {} -- (\x+60:1);
\draw[shift = {(\x:1)}] (0,0) -- (\x:0.3);
\draw[shift = {(\x:1)}] (0,0) -- (\x-60:0.3);
\draw[shift = {(\x:1)}] (0,0) -- (\x+60:0.3);
}
\foreach \x in {90, 270}
{
\draw (o1) -- (\x:1) node {};
\draw (o2) -- (\x:1) -- (\x+60:1);
\draw[shift = {(\x:1)}] (0,0) -- (\x:0.3);
\draw[shift = {(\x:1)}] (0,0) -- (\x-60:0.3);
\draw[shift = {(\x:1)}] (0,0) -- (\x+60:0.3);
}
\end{scope}
\end{tikzpicture}
\caption{The two Inoue structural induction moves.}
\label{fig:Inoue}
\end{figure}

It is tempting to try to apply our method to every cocompact right-angled reflection group in
 $\hyperbolic^3$ by affirmatively solving the following:
\begin{prob}\label{prob:Pogorelov Duals Win}
Does  every Pogorelov graph have a legal system?
\end{prob}
As we have verified that the dual L\"{o}bell graphs have legal systems in Lemma~\ref{lem:LobellGraphsWin},
in view of Inoue's result, it is conceivable that one might be able to approach Problem~\ref{prob:Pogorelov Duals Win} by structural induction. However, we caution that the corresponding problem for finite covolume reflection groups has a negative solution,
as indicated by  Theorem~\ref{thm:DWtoHam} and the examples discussed immediately afterwards.

\subsection{Characterizing finite volume 3-dimensional right-angled hyperbolic reflection groups}

Let $\Gamma\subset S^2$ be a graph embedded in the sphere.
We use the term \emph{quad} for a full 4-cycle in $\Gamma$ that bounds a region of $S^2$.

\begin{prop}[Cusped 3-dimensional hyperbolic reflection groups]\label{prop:right angled reflection}
Let $\Gamma\subset S^2$ be a simplicial connected graph embedded in the sphere.
Let $G=G(\Gamma)$ be the associated right-angled Coxeter group.
Then $G$ is isomorphic to a finite co-volume cusped hyperbolic reflection group precisely
if the following conditions hold:
\begin{enumerate}
\item\label{cuspitem:1} Each region is bounded by a quad or triangle (the quads generate the cusps).
\item\label{cuspitem:2} The intersection of distinct quads is either empty, a vertex, or an edge.
\item\label{cuspitem:3} Any cycle of length~$\leq 4$ bounds a region or the union of two triangular regions.
\item\label{cuspitem:4} $\Gamma$ is not a 
triangle, $4$-clique, $4$-cycle, or cone on a $4$-cycle.
\end{enumerate}
\end{prop}

A simple example is where $\Gamma$ is the 1-skeleton of a cube as discussed in Example~\ref{ex:cube}. It is dual to a right-angled hyperbolic ideal octahedron, which is the fundamental domain of the reflection group. Another simple example is the $1$-skeleton of a triangular prism.
\begin{com}it has an easy legal system with three moves:
each move has one vertex on each triangle.
a start state is a triangle.
one notes that the union of two vertices of different colors contains the third vertex in its 1-nbhood,
and that if they aren't connected then this third vertex connects them...
\end{com}

Note that $G$ is virtually abelian in each of the cases listed in Condition~\eqref{cuspitem:4}. In the presence of Conditions~\eqref{cuspitem:1}-\eqref{cuspitem:3}, Condition~\eqref{cuspitem:4} is equivalent to requiring that $\Gamma$ has at least 6 vertices. To establish this equivalence the reader can consider all 32 connected simplicial graphs with at most $5$ vertices.
\begin{com}
Most cases are excluded by having a region with at least 5 sides or 3 triangles that share an edge.
\end{com}

Note that the final two conditions together imply that  no vertex is surrounded by four or fewer triangles.

This is a highly simplified version of known results, and we are grateful to Igor Rivin for tracing these references for us. A classical reference is to Andre'ev's paper \cite{Andreev70} which was corrected by Hodgson \cite{Hodgson90} using \cite{RivinPhD}. Another proof was given by Hubbard-Roeder-Dunbar \cite{RoederHubbardDunbar2007} along Andre'ev's lines. Finally there is an orbifold proof by Thurston in
\cite{ThurstonNotes}.
We instead give an explanation in the context of geometric group theory, that also depends on hyperbolization:

\begin{proof}
We now show that if $\Gamma$ satisfies Conditions~\eqref{cuspitem:1} - \eqref{cuspitem:4} then $G$ is hyperbolic relative to its quads. Suppose $\Gamma$ contained a full subgraph of the form $B\star C = \{b_1,b_2\}\star \{c_1,c_2,c_3\}$
where $c_1,c_3$  and $b_1,b_2$ are nonadjacent.
Then $b_1c_1b_2c_3$
is a full 4-cycle
and is therefore a quad by Condition~\eqref{cuspitem:3}.
Hence neither
 $b_1c_1b_2c_2$ nor $b_1c_2b_2c_3$ are quads by Condition~\eqref{cuspitem:2}. Hence $c_1,c_2$ and $c_2,c_3$ are adjacent,  and hence $\Gamma$ is the $c_2$-cone on the quad $b_1c_1b_2c_3$ which violates Condition~\eqref{cuspitem:4}.
Observe that each full 4-cycle is a quad by Condition~\eqref{cuspitem:3}. It follows from Corollary~\ref{cor:rel hyp fin vol} that $G$ is hyperbolic relative to its quads.

We now verify that  $G$ is a finite volume hyperbolic reflection group.
Regard $G$ as acting by reflections on a simply-connected 3-manifold $\widetilde M$ with fundamental domain a ball
whose boundary is the cell structure on $S^2$ that is dual to $\Gamma$.
The stabilizer of each vertex is the right-angled reflection group associated to a region of $\Gamma$ and note that by Condition~\eqref{cuspitem:1} this is either $(\integers_2*\integers_2)^2$ or $\integers_2^3$.
Let $G'$ be a torsion-free finite index subgroup. The quotient $G'\backslash \widetilde M$ has the property that the link of each quad vertex is a torus, and we can thus remove a finite neighborhood from each such vertex to obtain
a manifold $\bar M$ whose boundary is a union of tori.
Note that  $G$ does not split along a finite group since $\Gamma$ is connected and has no internal $3$-cycles by Condition~\eqref{cuspitem:3}. \begin{com} If it split along a finite group, then it would have a finite index subgroup that split as a nontrivial free product. There would then be a sphere  or compression disk corresponding to the splitting, and this would contradict that the 3-manifold is aspherical or that the boundary of a core is incompressible.\end{com}
Note that $\bar M$ is not a ball or a thickened torus by the excluded degenerate possibilities for $\Gamma$.
Our peripheral structure consists entirely of the conjugates of subgroups commensurable with the $G(Q)$ where $Q$ varies over the quads, and thus the  JSJ decomposition for $\bar M$ cannot have a Seifert-fibered piece,
and must consist of a single piece for otherwise $\pi_1\bar M$ would split over $\integers^2$. We conclude that $M$ is a finite-volume hyperbolic manifold by Thurston's geometrization.

For the converse, assume that $G$ is a finite volume hyperbolic reflection group.
Condition~\eqref{cuspitem:1} holds since if there were some region with more than 4 sides,
 then $\euler(G')<0$ where $G'$ is a torsion-free finite index subgroup.
Condition~\eqref{cuspitem:2} holds since otherwise there would be maximal virtually $\integers^2$ subgroups
with infinite intersection.
Condition~\eqref{cuspitem:3} holds since otherwise $G$ would split over a virtually $\integers^2$ or a finite group associated to a full 4-cycle or 3-cycle that does not bound a region.
\begin{com}reference: finite volume hyperbolic reflection group does not split along a virtually $\integers^2$ or a finite group
\end{com}
 Condition~\eqref{cuspitem:4} holds since $G$ contains a rank two free subgroup but each of the degenerate cases is virtually abelian.
 \end{proof}


The \emph{join} $A\star B$ of two nonempty graphs, is the graph consisting of $A\sqcup B$ together with an edge from each vertex of $A$ to each vertex of $B$. A \emph{clique} is a complete graph $K(n)$ where $n\geq 0$.

The following is a special case of a characterization given by Caprace \cite[Thm~A${}'$]{Caprace09,Caprace13}:
\begin{thm}\label{thm:caprace}
Let $\mathcal T$ be a collection of subgraphs of $\Gamma$. Then $G(\Gamma)$ is hyperbolic relative to
$\{ G( K ) \ : \ K \in \mathcal T \}$ if and only if the following conditions are satisfied:
\begin{enumerate}[(a)]
\item\label{caprace:RH1} For each non-clique subgraphs $J_1,J_2$ whose join $J_1 \star J_2$ is a full subgraph of $\Gamma$ there exists $K\in\mathcal T$ such that $J_1\star J_2\subset K$.
\item\label{caprace:RH2} For all $K_1,K_2\in\mathcal T$ with $K_1\neq K_2$, the intersection $K_1\cap K_2$ is a clique.
\item\label{caprace:erratum} For all $K\in \mathcal T$ and nonadjacent vertices $v_1,v_2\in K$,
 if $v$ is adjacent to both $v_1$ and $v_2$ then $v\in K$.
\end{enumerate}
\end{thm}
\begin{proof}
The three conditions of Theorem~\ref{thm:caprace} are straightforward simplifications of the corresponding
three conditions in \cite[Thm~A']{Caprace13} except that Condition~\eqref{caprace:RH1} is stated
with a weaker assumption that  $J_1,J_2$ are neither joins nor cliques. Thus Theorem~\ref{thm:caprace}
 is implied by Caprace's original statement.

To see that our variant implies Caprace's original version we argue as follows:
First note that if $J_1$ is a join, then it can be expressed as a join $A_1\star B_1$
 where $A_1$ is not a join or a clique. If $J_2$ is not a join then we let $A_2=J_2$,
 and otherwise we let $J_2=A_2\star B_2$ where $A_2$ is not a join or a clique.
 Condition~\eqref{caprace:RH1} implies that  $A_1\star A_2$ lies in some $K\in \mathcal T$.
 Let $x_i,y_i$ be non-adjacent vertices of $A_i$.  For each vertex $v_i$ of $B_i$,
  Condition~\eqref{caprace:erratum} applied to $v_i$ with $x_i,y_i\in K$
shows that $v_i\in K$.
\end{proof}


\begin{cor}\label{cor:rel hyp fin vol} Let $\Gamma$ be a simplicial graph.
 The right-angled Coxeter group $G=G(\Gamma)$ is hyperbolic relative to
 $\{G(Q)\ : \ Q \text{ is a full 4-cycle}\}$
 if and only if $\Gamma$ has no full subgraph $B\star C$ where $|V(B)|=2$ and $|V(C)|=3$ and $B,C$ are not cliques.
 \end{cor}

\begin{proof}
Suppose $\Gamma$ has a full subgraph $B\star C$ as above.
Let $b_1,b_2$ be the vertices of $B$. Let $c_1,c_2,c_3$ be the vertices of $C$ and assume $c_1,c_3$ are nonadjacent. Then $\{b_1,c_1,b_2,c_3\}$ are the vertices of a full 4-cycle $Q$. Thus $c_2\in Q$ by Condition~\eqref{caprace:erratum} of
Theorem~\ref{thm:caprace}, which is impossible.

Let $\mathcal T$ consist of all full $4$-cycles.

We now verify Condition~\eqref{caprace:RH1}. Suppose $J_1,J_2$ are non-clique subgraphs whose join is a full subgraph of $\Gamma$. Since $J_1,J_2$ are not cliques, they each have at least two vertices, and neither consists of a single edge.
If they both have two vertices then $J_1\star J_2$ is a full 4-cycle.
Otherwise, one contains a full subgraph $A$ consisting of two non adjacent vertices,
 and the other contains a full subgraph $B$ consisting of three vertices that is not a clique.
 Then $A\star B$ is a full subgraph of $\Gamma$ which contradicts our hypothesis.

If one of $K_1,K_2$ in Condition~\eqref{caprace:RH2} is a pair of non adjacent vertices then $|V(K_1\cap K_2)|\leq1$. Suppose $K_1,K_2$ are both $4$-cycles.
We will show that $K_1\cap K_2$ does not contain a pair of nonadjacent vertices, and so $K_1\cap K_2$ consists of a single edge or vertex, and is thus a clique.
 If $K_1\cap K_2$ contains a pair $u,v$ of nonadjacent vertices and let $B$ be the full graph spanned by $u,v$, and let $C$ be the graph spanned by the vertices of $K_1,K_2$ not in $B$. Note that $V(C)\geq 3$ since otherwise $K_1=K_2$. Since $C$ contains a pair of opposite vertices of each $K_i$ we see that $C$ is not a clique.
  Since $u,v$ are adjacent to each vertex of $C$ we see that $B\star C$ is a full subgraph of $\Gamma$ which contradicts our hypothesis.

We now verify Condition~\eqref{caprace:erratum}.
Suppose that $v$ is adjacent to vertices $v_1,v_2$,
and that $v_1,v_2$ are in some $K\in \mathcal T$ and $v_1,v_2$ are not adjacent.
Then either $v\in K$, or letting $v_3,v_4$ be the other vertices of $K$,
and letting $B,C$ be the full subgraphs whose vertices are $\{v_1,v_2\}$ and $\{v,v_3,v_4\}$
we see that $B\star C$ is a full subgraph of $\Gamma$, which is impossible.

\end{proof}

\section{Examples where the method fails}\label{sec:fails}
In Sections~\ref{sub:coning off} and \ref{sub:duals} we describe examples of graphs $\Gamma$ such that $\curvature(\Gamma)\geq 0$
but such that there is no legal system, because $\Gamma$ does not have a single legal state. We emphasize that not only does our method fail for these examples, but Bestvina-Brady Morse theory cannot be successfully applied to any finite cover of $X$.
The first class of examples are easy bipartite graphs described in Section~\ref{sub:coning off}.
In Section~\ref{sub:duals} we describe a second class of examples that are more complicated but have the advantage that they are planar.
In Section~\ref{sub:negative euler} we show that if $\Gamma$ has a legal system then
$\curvature_2(\Gamma) =1-v/2+e/4 \geq  0$.

\subsection{Bipartite Cones}\label{sub:coning off}
Consider the bipartite graph $\Lambda_{(m,k)}$ whose vertices are $N_m\sqcup C(m,k)$ where  $N_m = \{1,\ldots, m\}$
and $C(m,k)$ consists of all k-element subsets of $N_m$.
An edge joins a vertex $r\in N_m$ to each element of $C(n,k)$ containing $r$.

\begin{prop}
For $m\geq 2k-1$ and $k\geq 2$ the graph $\Lambda_{(m,k)}$ does not have a legal state.
\end{prop}
\begin{proof}
Consider a state $S$.
By the pigeon-hole principle, at least $k$ elements of $N_m$ are either in $S$ or in $N_m-S$.
Without loss of generality assume the former.
Let $v$ be the vertex of $C(m,k)$ that is joined to these $k$ elements of $N_m$.
Either $S=V-\{v\}$ or $m_vS=V-\{v\}$ since otherwise either $S$ or $m_vS$ is not legal since $v$ is separated by its neighbors.
Without loss of generality assume the former.
For each $u\in C(m,k)-\{v\}$ we have $m_u = \{u,v\}$ since otherwise $m_uS$ is not legal as
$u,v$ lie in different components.
Finally, note that for any two distinct $u_1,u_2\in C(m,k)-\{v\}$ the state $m_{u_1}m_{u_2}S = \{u_1,u_2,v\}$ is not legal.
\end{proof}

We note that $\Lambda_{(m,k)}$ is connected for $m\geq k$ and $\girth(\Lambda_m)=4$ for $m\geq 4$.
We have $\curvature(\Lambda_{(m,k)})>0$ for $m\geq 2k-1$ since
 $\curvature(\Lambda_{(m,k)})= 1- (m + {m \choose k})/2 + (k{m \choose k})/4
= 1-m/2 + (k/4-1/2){m \choose k}$.
The graph $\Lambda_{(4,3)}$ is the 1-skeleton of the 3-cube, and $\curvature(\Lambda_{(4,3)})=0$.
However, as above, $\Lambda_{(5,3)}$ provides a very small example of a graph with no legal state
but with $\curvature>0$.

\begin{exmp}
We will show that the bipartition of $\Lambda_{(m,k)}$ provides a legal system when $3<k<m<2k-1$.
We need to find sets $S_N\subset N_m$ and $S_C\subset C(m,k)$ such that $S_N\cup S_C$ is a state whose orbit is legal. Let $S_N= \{1,\ldots, \lceil{m/2}\rceil \} \subset N_m$. Note that $|N_m- S_N|< k$ so every $k$-subset of $N_m$ contains at least one element of $S_N$, i.e. every vertex in $C(m,k)$ is joined to a vertex in $S_N$. Since $k> m/2$ and $k< m$ there are at least two distinct $k$-subsets  of $N_m$ containing $S_N$, i.e. two vertices $v,v'\in C(m,k)$ both joined to all vertices of $S_N$. Similarly, there are vertices $u,u'\in C(m,k)$ both joined to all vertices of $N_m-S_N$. Clearly $v,v',u,u'$ are all distinct, since no vertex of $C(m,k)$ is joined to all vertices on $N_m$. Letting $S_C = \{v,u\}$ we obtain a state $S_N\cup S_C$ whose orbit is legal.
\end{exmp}

\subsection{Vertex-face-incidence-graphs  and hamiltonian cycles}\label{sub:duals}

 A \emph{cycle} $C$ in $\Theta$ is a subgraph homeomorphic to a circle.
The cycle $C$ is \emph{hamiltonian} if each vertex of $\Theta$ lies in $C$.
  Given an embedding $\Theta\subset S^2$ in the sphere,
   we let $F$ denote the set of the \emph{faces} which are components of $S^2-\Theta$.

  The \emph{vertex-face-incidence graph} $\VF(\Theta)$ is a bipartite graph whose vertices are partitioned as
  $V\sqcup F$ and where $v \in V$ is joined to $f \in F$ by an edge if and only if $v\in \boundary f$.

\begin{defn}[$k$-connected]
$\Theta$ is \emph{$k$-connected} if there does not exist a subset $U\subset V=\Theta^0$ of at most $(k-1)$ vertices such that the subgraph induced by $V-U$ is disconnected.

For instance, if $\Theta$ is $3$-connected then an embedding $\Theta\subset S^2$ is essentially unique by Steinitz's theorem.
\end{defn}

    The following is easy to verify:
\begin{prop} Let $\Theta$ be a $2$-connected planar simplicial graph. Then $\VF(\Theta)$ has a planar embedding
such that each face of $\VF(\Theta)$ is a quadrangle, corresponding to an edge of $\Theta$. Moreover,
$\VF(\Theta)$ is $3$-connected if and only if $\Theta$ is $3$-connected.
\end{prop}

The main result of this section is:

\begin{thm}\label{thm:DWtoHam}
Let $\Theta$ be a  $2$-connected $3$-valent planar graph. Then the following are equivalent:
\begin{description}
\item [(a)] $\Theta$ has a Hamilton cycle,
\item [(b)] $\VF(\Theta)$ has a strongly legal state,
\item [(c)] $\VF(\Theta)$ has a legal system.
\end{description}
\end{thm}

Figure~\ref{f:K4vsCube} illustrates the correspondence between a Hamilton cycle in $\Theta$
and a strongly legal state of $\VF(\Theta)$ in the case when $\Theta$ is a ``house''.

\begin{figure}
                        \newcommand\cir{node [circle,draw,fill,inner sep=1pt]}
\[\begin{tikzpicture}[baseline=0, scale = 0.55]
\draw[black,thick]
{
	(0,2) \cir {} -- (0,0) \cir {} -- (2,0) \cir {} -- (2,2) \cir {} -- (1,3.41) \cir {} -- (0,2) -- (2,2)
};
\draw[red, thick]
{
	(0,0) -- (2,0) -- (2,2) -- (1,3.41) -- (0,2) -- (0,0)
};
\node at (1,-.5) {$\Theta$};
\end{tikzpicture}
\;\;\;\;\;\;
\begin{tikzpicture}[baseline=0, scale = 0.55]
\draw[black,thick]
{
	(0,2) \cir {} -- (0,0) \cir {} -- (2,0) \cir {} -- (2,2) \cir {} -- (1,3.41) \cir {} -- (0,2) -- (2,2)
};
\node at (1,-.5) {$\Theta$};
\end{tikzpicture}
\;\;\;\;
\begin{tikzpicture}[baseline=0, scale= 0.55]
\draw[Thistle]
{
	(1,1) \cir {}
	(1,2.5) \cir {}
	(-1,4) \cir {}
};
\draw[Thistle]
{
	(0,2)  to (1,1)  -- (0,0)
	(2,0)  -- (1,1) -- (2,2)  -- (1,2.5)  -- (1,3.41)
	(0,2) -- (1,2.5)
	(-1,4)  to (1,3.41)
	(-1,4) to (0,2)
	(-1,4) to (0,0)
	(-1,4) to[out=260,in= 180] (0,-.5) to[out=0, in=190] (2,0)
	(-1,4) to[out=30, in=60] (2,2)
};
\draw[black,thick]
{
	(0,2) \cir {} -- (0,0) \cir {} -- (2,0) \cir {} -- (2,2) \cir {} -- (1,3.41) \cir {} -- (0,2) -- (2,2)
};
\end{tikzpicture}
\begin{tikzpicture}[baseline=0,scale=0.55]
\draw[black,thick]
{
	(0,2) \cir {} to (1,1) \cir {} -- (0,0) \cir {}
	(2,0) \cir {} -- (1,1) -- (2,2) \cir {} -- (1,2.5) \cir {} -- (1,3.41) \cir {}
	(0,2) -- (1,2.5)
	(-1,4) \cir {} to (1,3.41)
	(-1,4) to (0,2)
	(-1,4) to (0,0)
	(-1,4) to[out=260,in= 180] (0,-.5) to[out=0, in=190] (2,0)
	(-1,4) to[out=30, in=60] (2,2)
};
\node at (1,-1) {$\mathcal{VF}(\Theta)$};
\end{tikzpicture}
\begin{tikzpicture}[baseline=0,scale=0.55]

\draw[black,thick]
{
	(0,2) \cir {} to (1,1) \cir {} -- (0,0) \cir {}
	(2,0) \cir {} -- (1,1) -- (2,2) \cir {} -- (1,2.5) \cir {} -- (1,3.41) \cir {}
	(0,2) -- (1,2.5)
	(-1,4) \cir {} to (1,3.41)
	(-1,4) to (0,2)
	(-1,4) to (0,0)
	(-1,4) to[out=260,in= 180] (0,-.5) to[out=0, in=190] (2,0)
	(-1,4) to[out=30, in=60] (2,2)
};
\draw[WildStrawberry,thick]
{
	(0,0) \cir {} to (1,1) \cir {} -- (0,2) \cir {} -- (1,2.5) \cir {} -- (2,2) \cir {} -- (1,1)
};
\draw[LimeGreen, thick]
{
	(1,3.41) \cir {} -- (-1,4) \cir {} to[out=260,in= 180] (0,-.5) to[out=0, in=190] (2,0) \cir {}
};
\node at (1,-1) {$\mathcal{VF}(\Theta)$};
\end{tikzpicture}
\]
\caption{ $\Theta$ is on the left and $\VF(\Theta)$ on the right.
 The hamiltonian cycle on the left corresponds to the strongly legal state on the right.}
\label{f:K4vsCube}
\end{figure}
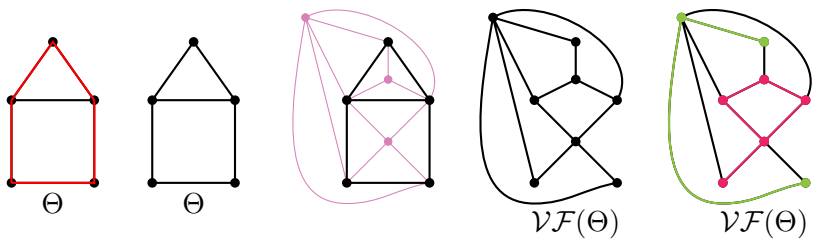

\begin{proof}[Proof of (c) $\Rightarrow$ (b)]
This follows from the definitions.
\end{proof}

\begin{proof}[Proof of (b) $\Rightarrow$ (a)]
Let $A$ be a legal state and let $B \ = \ (\VF(\Theta))^0-A$.
 Let $E \subseteq E(\Theta)$ be the set of edges $e$ of $\Theta$ for which the two faces meeting along
  $e$ belong to different parts of $A\sqcup B$. By definition of a strongly legal state, every vertex of $\Theta$ is incident with faces in both $A$ and $B$. Thus as $\Theta$ is $3$-valent we find that every vertex of $\Theta$ is incident to exactly two edges of $E$. Thus $E$ is the edge set of collection of disjoint cycles $\{C_i\}$ in $\Theta$. We will show that this collection consists of a single cycle that is a hamiltonian cycle of $\Theta$.
  Suppose this were not the case. Let $R$ be the set of regions of $S^2- \bigcup_i C_i$.
  Note that any two faces of $\Theta$ that lie in the same region belong to the same part of $A\sqcup B$.
If there were two or more cycles then there would be three or more regions and hence two whose faces are in the same part, say $A$. A path in $A$ from one to the other would have to pass through a region whose faces are in $B$.
\end{proof}

\begin{proof}[Proof of (a) $\Rightarrow$ (c)]
 Let $C$ be a Hamiltonian cycle of $\Theta$.
 Since $\Theta$ is 3-valent, by choosing an arbitrary way of directing the edges of $\Theta-C$, we find  that each vertex of $\theta$ is either an initial or terminal vertex.
We thus have a partition $V(\Theta)=V_1\sqcup V_2$ consisting of initial and terminal vertices.

We will apply Lemma~\ref{lem:TBWS} where $A=F$ and $B=V$ denote the face-vertices and vertex-vertices of $\VF(\Theta)$.
Let $F_1,F_2$ consist of the faces on opposite sides of $C\subset S^2$.
Let $V_1,V_2$  be  $A$ and $\Theta^0-A$.

Condition~\eqref{TBWS:1} is satisfied as follows:
Consider the disk diagram $D_i$ consisting of a hemisphere bounded by $C$,
so the faces of $D_i$ are the elements of $F_i$.
For each pair of faces in $F_i$, there is a gallery joining them in $D_i$
that consists of a sequence of faces meeting along common edges.
Since each such edge is in the interior of $D_i$, it is an edge of $\Theta-C$.
Therefore it has one vertex in $Y_1$ and the other vertex in $Y_2$.
Thus this gallery provides the desired path.

Condition~\eqref{TBWS:2} is satisfied since every vertex lies on $C$ and thus lies on a face in each hemisphere.
\end{proof}

Theorem~\ref{thm:DWtoHam} implies that characterizing the planar quadrangulations (planar graphs in which every face has four edges) having legal systems is a difficult task. In particular a classical conjecture of Barnette states that every bipartite $3$-connected $3$-regular planar graph has a Hamilton cycle (see \cite{Gould2014}). Using Theorem~\ref{thm:DWtoHam} we can awkwardly restate this conjecture in the language of legal systems.

\begin{conj}[Barnette]\label{conj:reformed barnette}
Let $Q$ be a $3$-connected planar quadrangulation, and let $(A,B)$ be the bipartition of the vertices of $Q$. Suppose that every vertex of $A$ has valence three, and that $A$ admits a partition $(A_1,A_2)$ such that every face of $Q$ is incident with a vertex in $A_1$ and with a vertex in $A_2$. Then $Q$ has a legal system.
\end{conj}

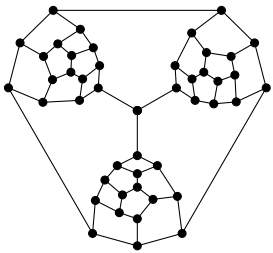
\begin{figure}
\tikzstyle{every node}=[circle, draw, fill,
                        inner sep=0pt, minimum width = 3pt]
\begin{tikzpicture}[scale = 0.6]
\foreach \x in {30, 150, 270}
{
	\draw (0,0) node {} -- (\x:1) node {};
	\draw (\x-20:1.3) node {} -- (\x:1) -- (\x+20:1.3) node {} -- (\x:1.4) node {} -- (\x-20:1.3) -- (\x-25:1.7) node {} -- (\x-10:1.9) node {} -- (\x:1.7) node {} -- (\x+10:2) node {} -- (\x+25:2.1) node {} -- (\x+20:1.3);
	\draw (\x:1.7) -- (\x:1.4);
	\draw (\x-25:1.7) -- (\x-25:2.2) node {} -- (\x-10:2.3) node {} -- (\x-10:1.9);
	\draw (\x-25:2.2) -- (\x-20:2.9) node {};
	\draw (\x+25:2.1) -- (\x+20:2.9) node {} -- (\x:3.0) node {} -- (\x-20:2.9);
	\draw (\x-10:2.3) -- (\x:2.4) node {} --(\x:3);
	\draw (\x:2.4) -- (\x+10:2);
	\draw (\x+20:2.9) -- (\x+100:2.9);
} 	
\end{tikzpicture}
\caption{Tutte's Graph is a famous planar graph with no hamiltonian cycle.}
\label{fig:Tutte}
\end{figure}

\begin{exmp}[Tutte]\label{exmp:tutte}
Tait had conjectured in 1884 that every $3$-valent $3$-connected planar graph has a hamiltonian cycle.
In 1946, Tutte produced a counterexample to this conjecture:
 the graph $\Theta$ depicted in Figure~\ref{fig:Tutte}.
Its vertex-face incidence graph $\VF(\Theta)$ has the property that any proper connected subgraph with $\curvature\geq 0$ is a square or a chain of length $\leq 2$.
\begin{com}
Indeed, a subgraph with $\kappa\geq 0$ would have all regions squares. So it suffices to check that $\VF(\Theta)$ does not contain a $4$-cycle bounding a region. This would correspond to a pair of regions in $\Theta$ with disconnected intersection.
\end{com}
Combining this with Theorem~\ref{thm:DWtoHam}, we deduce that there exist 3-connected planar quadrangulations  not admitting a legal state, and hence Guess~\ref{guess:naive} is false for planar graphs, as there is not even a single legal state.
Consequently, Bestvina-Brady Morse theory cannot be applied to any  cover of the cube complex $X$ associated to a torsion-free finite index subgroup of $G$.

The Coxeter group $G=G(\VF(\Theta))$ corresponds to the fundamental domain of a finite co-volume cusped 3-dimensional hyperbolic right-angled Coxeter group. Indeed, there is no internal 3-cycle in $\VF(\Theta)$ since it is bipartite and no internal 4-cycle in $\VF(\Theta)$ since any two regions of $\Theta$ have connected intersection - as revealed by an inspection of Figure~\ref{fig:Tutte}. The intersection of two quads is either empty, a vertex, or an edge, since otherwise either $\Theta$ is not simplicial, or two regions of $\Theta$ overlap along more than one edge. Consequently,  $\VF(\Theta)$ satisfies the conditions of Proposition~\ref{prop:right angled reflection}.
\end{exmp}

\begin{com}
Does the pyramidal triangulation of a quadrangulation of the 2-sphere always have a legal system?
Always have a legal state?

Does there exist a flag triangulation of $S^2$ that has no legal state?
Any such should arise as a vertex link of some cubulated hyperbolic 3-manifold.
\end{com}

\begin{rem}\label{rem:connectivity}
Tutte proved that $4$-connected planar graphs are hamiltonian,
 and so perhaps $\VF(\Theta)$ has a legal state when $\Theta$ has stronger connectivity properties:
 e.g. in the spirit of Conjecture~\ref{conj:reformed barnette}.
 Finally, we note that
 Grinberg's formula \cite{Grinberg68} provides a host of other counterexamples to Tait's conjecture,
and these provide a rich family of examples of 3-dimensional right-angled hyperbolic reflection groups where
Bestvina-Brady Morse theory cannot possibly show virtual fibering for any finite cover of the associated cube complex.

The above examples show that, using the standard affine structure on the cubes, Bestvina-Brady theory cannot be applied to the cube complex which is the dual spine
 to the hyperbolic tiling by ideal polyhedra that is formed from the reflection walls of certain hyperbolic Coxeter groups. In Example~\ref{exmp:barycentric} we show that the method of this text cannot be applied to certain right-angled Coxeter groups that are commensurable with closed (non-hyperbolic) 3-manifold fundamental groups.
 \end{rem}

\begin{rem}
Consider an embedding $\Theta\subset S^2$.
As $\girth(\VF(\Theta)) = 4$ we have
$$\kappa(\VF(\Theta)) = 1 - \frac{|V(\VF(\Theta))|}{2} + \frac{|E(\VF(\Theta))|}{4}= 1 - \frac{|V(\Theta)| + |F(\Theta)|}{2} + \frac{2|E(\Theta)|}{4} = 0$$
where the final equality holds by Euler's formula for $\Theta\subset S^2$.
\end{rem}

\begin{lem}\label{lem:conesquare}
 Let $C$ be a cycle of length four in a graph $\Gamma'$, and let  $\Gamma$ be the graph obtained from $\Gamma'$ by adding a vertex adjacent to the vertices of $C$ and to no other vertices. If $\Gamma$ has a legal system then so does $\Gamma'$.
\end{lem}

\begin{rem}
The analogous statement holds more generally with $C$ replaced by
the cocktail party graph $K(2,2,\ldots, 2)$.
More generally, consider the amalgam $\Gamma = \Gamma'\cup_C \Gamma''$ of the graphs $\Gamma',\Gamma''$ along such $C$:
If $\Gamma$ has a legal system, then so do $\Gamma'$ and $\Gamma''$. This specializes to the above result when
 $\Gamma''$ is the cone on a cocktail party graph $C$.
\end{rem}
\begin{proof}
Let $a,b,c,d$ be the vertices of $C$ in order, and let $v$ be the new vertex. Let $M$ be the group associated to the legal system on $\Gamma$, let $O$ be a legal $M$-orbit, let $V$ denote the vertex set of $\Gamma$,
and let $V'=V-\{v\}$ be the vertex set of $\Gamma'$.
 We claim that the restriction of $M$  and $O$ to $\Gamma'$ (i.e. ignore $v$) provides a  legal system.

Suppose not. Then w.l.o.g. there exists a state $S \in O$ such that $S \cap V'$ does not induce a connected subgraph of $\Gamma'$. As $S$ induces a connected subgraph in $\Gamma$, w.l.o.g. we have $a,c,v \in S$ and $b,d \not \in S$. However, there exists $m \in M$ such that $a,c \in m$  but $v,b,d \not \in m$. Indeed, either $m_a$  or $m_c$ or $m_a+m_c$ has this property.
The state $mS$ is not strongly legal since $v \in mS$, but  $a,b,c,d \not \in mS$.
\end{proof}

\begin{exmp}\label{exmp:barycentric}
Let $\Gamma$ be the 1-skeleton of the first barycentric subdivision of the cell structure for $S^2$ whose 1-skeleton is Tutte's graph. Then the clique complex of $\Gamma$ is $S^2$ but $\Gamma$ has no legal system.
Indeed, by starting with $\Gamma$ and applying Lemma~\ref{lem:conesquare} sixty-nine times (once for each edge of $\Theta$) we arrive at $\VF(\Theta)$. Hence $\Gamma$ has no legal-system since $\VF(\Theta)$ has no legal-system by Example~\ref{exmp:tutte}.
\end{exmp}

\subsection{Failure With Negative Euler characteristic}\label{sub:negative euler}
\begin{rem}When $G$ is a locally quasiconvex hyperbolic group it cannot have an infinite index f.g. infinite normal subgroup \cite[Prop~3.9]{ABC91}. Likewise, a non virtually abelian Kleinian group that is locally geometrically finite cannot have an infinite index nontrivial f.g. normal subgroup.
There are thus a variety of examples of right-angled Coxeter groups $G(\Gamma)$ where $\Gamma$ has no legal system.

A prominent such family arise from planar graphs $\Gamma\subset S^2$ such that some region has more than five sides
but where each cycle of length~$3$~or~$4$ bounds a region. In this case, $G(\Gamma)$ has an index~$\leq 16$ subgroup that is the fundamental group of a hyperbolic 3-manifold $M$ where $\boundary M$ contains a surface of genus~$>1$ that corresponds to the large region. Thus $\pi_1M$ is locally geometrically finite by a result of
Thurston~\cite{Canary96}, and hence cannot have a nontrivial normal subgroup that is finitely generated.

The previous class of examples has $\euler(G)<0$.
More generally, other examples of such Coxeter groups are virtually 2-dimensional coherent groups where $\euler(G)\neq 0$. \begin{com}There are many such examples in \ref{SectionalpqrITHINK}.\end{com}
For them the f.g. kernel of Corollary~\ref{cor:win f.g.}  would be free and thus $\euler(G)=0$, which is impossible.
Indeed,  if  $N\subset G$ is a nontrivial f.p. infinite index normal subgroup with  $\cd(G)=2$
then $N$ is  free \cite{Bieri76}.
 \end{rem}

The above discussion suggests that $\beta_0(G)-\beta_1(G)+\beta_2(G) \geq 0$ when there is a legal system
and $\cd(G)=2$,
and indeed, we have the following simple and precise count that holds in general:
\begin{thm}
Let $\Gamma$ be a finite graph. Suppose there is a legal system for $\Gamma$.
Then $\kappa_2(\Gamma) = 1-\frac{V}{2}+\frac{E}{4} \geq 0$.

Moreover, if $\kappa_2(\Gamma) = 0$ then all states in a legal system are trees.
\end{thm}
\begin{proof}
Let $K$ be a $d$-clique in $\Gamma$. Observe that $K$ occurs in exactly $\frac1{2^d}$ of the states in the legal orbit. Indeed, letting $m_1,\ldots, m_d$ denote the moves at the $d$ vertices of $K$,
we see that the orbit is partitioned into  cardinality $2^d$ equivalence classes  according to the action of $\langle m_1,\ldots, m_d\rangle$ and $K$ appears in precisely one element within each class.
As each state $S$ is connected we have  $0\leq 1-v(S)+e(S)$.
Thus letting $n$ denote the cardinality of the legal orbit we have:
$$0\leq \sum_S (1-v(S)+e(S))
 = \sum_S 1  - \sum_S v(S)  + \sum_S e(S)
=  n - n \frac12 V + n\frac14 E
=n(1-\frac{V}{2}+\frac{E}{4})$$
Moreover, the above inequality is an equality precisely if each state is a tree.
\end{proof}
The above proof works for any inequality uniformly satisfied by numbers of cells for each state.
In particular, it is of interest in the case where we assume that each state is acyclic, and we have the following
which is also a consequence of a  cohomology computation:
\begin{cor}\label{cor:contractible states}
Suppose $\Gamma$ has a legal system with the property that the flag complex $Q(S)$ is contractible for each state $S$. Then
the generalized form of the above inequality provides that $\curvature(\Gamma)=0$
and hence $\euler(G)=0$. (Here, $\euler(G)=\frac{1}{[G:G']}\euler(G')$ where
$G'$ is a finite index torsion-fee subgroup.)
\end{cor}

\section{In pursuit of an exotic subgroup of a hyperbolic group}\label{sec:7}
In this section we pose a problem aiming to use Theorem~\ref{thm:win fiber} to provide a
hyperbolic group $G(\Gamma)$ such that the kernel $K$ of $G'\to \integers$ has finite $K(\pi,1)$ but is not hyperbolic.
\begin{prob}
Find $\Gamma$ having a legal system whose states have contractible clique-complexes,
such that the clique-complex $Q(\Gamma)$ is flag-no-square but does not contain a 2-sphere.
And such that $\vcd(G(\Gamma))\geq 3$. \begin{com}I do not yet know of a quick way to see that $\cd(G)>2$ by looking at $\Gamma$. I presume it suffices to see that $\homology_2(Q(\Gamma))\neq 0$ or perhaps that
$Q(\Gamma)$ does not deformation retract to a graph.\end{com}
\end{prob}

We do not know how to make a flag-no-square complex that is a closed pseudo-manifold of dimension $\geq3$ for which there is a legal system. For that matter we do not know how to make a flag-no-square (closed pseudo-manifold) of $\cd(Q(\Gamma))\geq 3$ with $\curvature(\Gamma)=0$. Note that if $\Gamma$ has a legal system and the clique complex of each state is contractible then $\curvature(\Gamma)=0$ by Corollary~\ref{cor:contractible states}.
In this connection we note that \cite[Conj~6.1]{LutzNevo14} implies that there does not exist a flag-no-square simplicial structure on $S^3$ with $\curvature(\Gamma)=0$.
\begin{com} and there is no flag-no-square structure on $S^n$ for $n>3$ or $n>4$
reference in Przytycki-Swiatkowski? or Davis-Jankiewicz?\end{com}

\begin{prob}
Find a finite simplicial graph $\Gamma$ with the property that:
\begin{enumerate}
\item	$\curvature(\Gamma)\leq 0$  (we are primarily interested in $\curvature(\Gamma)=0$.)
\item 	The clique-complex of $\Gamma$ is flag-no-square  (i.e. any length 4 cycle bounds two triangles.)
\item $\Gamma$ is not planar
\item Every edge of $\Gamma$ lies on a triangle (preferably at least two).
\end{enumerate}
\end{prob}

\begin{prob}
Is there $\Gamma$ as above with the more general property that no two $4$-cycles (without triangles) share non-adjacent vertices? And such that there is no 2-sphere in the space obtained by filling in all 3-cycles and 4-cycles?

It might be possible to find a legal system for such $\Gamma$ and use that to create a word-hyperbolic group with a non-hyperbolic normal subgroup with strong finiteness properties.
\end{prob}

\section{Random graphs}\label{sub:random}\label{sec:8}

In this section we show that Erd\H{o}s-R\'enyi random graphs for a wide range of densities almost surely contain legal systems. Recall that the model $\mcl{G}(n,p)$ consists of graphs with vertex set $V=\{1,2,\ldots,n\}$ in which edges are chosen independently with probability $p$. We are interested in the asymptotic behaviour of the model and so assume $n \to \infty$ throughout.

We will show that if $p$ is reasonably far away from both zero and one, then a graph in $\mcl{G}(n,p)$ almost surely has a legal system.
While our bounds on $p$ are not the best possible, they significantly imply that almost every labelled graph on $n$ vertices has a legal system, as the model $\mcl{G}(n, 1/2)$ corresponds to the uniform distribution on such graphs.
The \emph{complement} of a graph $\Gamma(V,E)$ is the graph with vertex set $V$ and where two distinct vertices $v,w$ are joined by an edge if and only if there is no edge between them in $\Gamma$.

A \emph{matching} in a graph $\Gamma(V,E)$ is a set $M\subseteq E$ such that no vertex in $V$ is incident to more than one edge of $M$. A \emph{perfect matching} in $\Gamma(V,E)$ is a matching $M$ such that every vertex of $V$ is incident to an edge of $M$.
We will need the following classical result about the existence of perfect matchings in random graphs.

\begin{thm}[Erd\H{o}s-R\'enyi~\cite{ERMatching66}]\label{thm:Gnp1}
	Let $p =(\log{n}+\omega(n))/n$ where $\omega(n) \to \infty$.  If $n$ is even then a graph in $\mcl{G}(n,p)$ almost surely contains a perfect matching.
\end{thm}	

Further, we will need a result on connectivity of random bipartite graphs.
Let $\mcl{G}(n_1,n_2,p)$ be a random graph model consisting of bipartite graphs with bipartition $(A,B)$ for a pair of disjoint sets $A$ and $B$, with $|A|=n_1$ and $|B|=n_2$, where the edges joining $A$ and $B$ are chosen independently with probability $p$.

\begin{thm}[Pal{\'a}sti~\cite{Palasti64}]\label{thm:Gnp2}
Let $0\leq n_1 \leq n_2$, and let $p =(\log{n_2}+\omega(n_2))/n_1$ where $\omega(n) \to \infty$. Then a graph in $\mcl{G}(n_1,n_2,p)$  is almost surely connected.
\end{thm}

\begin{thm}[Bollob{\'a}s~\cite{Bollobas80}]\label{thm:maxdegree}
The maximum valence of a graph in $\mathcal G(n,p)$ is almost surely $O(pn)$.
\end{thm}

Let $\Gamma=(V,E)$ be a graph, a \emph{perfect antimatching} is   a partition $\mcl{M} = \{S_1,\ldots,S_k\}$ of $V$ such that each $S_i$ consists of precisely two nonadjacent vertices.
  A set $T \subseteq V$ is an \emph{$\mcl{M}$-transversal} if  $|T \cap S_i|=1$ for each $i$.

A perfect antimatching $\mcl{M}$ is \emph{lawful} if every $\mcl{M}$-transversal is a legal state.
Note that the set of all $\mcl{M}$-transversals forms an orbit of  the subgroup of $2^{V}$ generated by  $\mcl{M}$.
Thus if  $\mcl{M}$ is lawful then it is a legal system.

\begin{thm}\label{thm:random} Suppose
\begin{equation}\label{inequalityp1} \frac{(2\log{n})^{1/2}+\omega(n)}{n^{1/2}} \leq p \leq  1-\frac{1}{n^{1.99}}
\end{equation}
for some $\omega(n)$ with $\omega(n) \to \infty$.
Then a graph in $\mcl{G}(n,p)$ almost surely contains a legal system.	
\end{thm}

Theorem~\ref{thm:random} is proven using colored legal system with at most two vertices of each color.
We hoped that by considering
colored legal systems with color classes of higher cardinality  one could decrease the lower bound. And indeed, we refer
the reader to Theorem~\ref{PGK} below:

\begin{proof}
The proof is broken into two cases that employ various types of legal systems.
The case where
$p \leq  1-\frac{2\log{n}+\omega(n)}{n}$ is treated in Lemma~\ref{lem:intermediate}.
The case where $ p \geq 1 -\frac{2\log{n}+\omega(n)}{n}$ is treated in Lemma~\ref{lem:high} under the assumption that $\Gamma$ is not a complete graph. Note that that $\Gamma$ is almost surely not complete if $p\leq 1-\frac{1}{n^{1.99}}$  since  the probability that $\Gamma$ is a complete graph equals $p^{\frac{n(n-1)}{2}}\leq e^{-(1-p)\frac{n(n-1)}{2}}\to 0$ as $n\to \infty$.
\end{proof}

\begin{lem}[Intermediate probability]\label{lem:intermediate}
Suppose $p$ satisfies the following inequality
\begin{equation}\label{inequalityp2}\frac{(2\log{n})^{1/2}+\omega(n)}{n^{1/2}} \leq p \leq  1-\frac{2\log{n}+\omega(n)}{n}
\end{equation}
for some $\omega(n)$ with $\omega(n) \to \infty$. Then a graph $\Gamma$ in $\mcl{G}(n,p)$ almost surely contains a legal system.
\end{lem}
\begin{proof}
We will first work under the assumption that $n$ is even, in which case we actually show that under Condition~\eqref{inequalityp2}, a random graph in $\mcl{G}(n,p)$ almost surely contains a lawful perfect antimatching.
We will adjust the argument to handle the case when $n$ is odd at the end of the proof.

We start by partitioning $V=\{1,2,\ldots,n\}$ into two even cardinality parts $A$ and $B$ each of size at most $n/2+1$. Since a perfect antimatching corresponds to a perfect matching of a complement, by the right inequality of Equation~\eqref{inequalityp2} and Theorem~\ref{thm:Gnp1} the subgraphs of $\Gamma$ induced by both $A$ and $B$ almost surely contain perfect antimatchings. We assume that  such antimatchings exist and  denote them by $\mcl{M}_A$ and $\mcl{M_B}$.

We will show that $\mcl{M} = \mcl{M}_A \cup \mcl{M_B}$ is almost surely lawful. Let $E_1$ be the set of edges with one end in $A$ and the other in $B$. Note that each edge in $E_1$ is chosen independently with probability $p$.

We construct an auxiliary bipartite graph $\Lambda'_A$ with bipartition $(A,\mcl{M}_B)$  by joining $u \in A$ and  $S \in M_B$ by an edge if $u$ is adjacent to both elements of $S$. Then $\Lambda'_A$ is a random bipartite graph in $\mcl{G}(|B|/2,|A|,p^2)$ and by  Theorem~\ref{thm:Gnp2} it is almost surely connected. Define $\Lambda_B'$ in $\mcl{G}(|A|/2,|B|,p^2)$ analogously.

We will show that if  $\Lambda'_A$ and  $\Lambda_B'$  are both connected then $\mcl{M}$ is lawful. Indeed, let $T$ be an $\mcl{M}$-transversal and let $\Gamma_T$ be the subgraph of $\Gamma$ induced by $T$. Consider $u,v \in T \cap A$, and let $P$ be a path in  $\Lambda'_A$ with ends $u$ and $v$. Replacing the vertices of $P$ in $\mcl{M}_B$ by corresponding elements of $T$, we obtain a path in $\Gamma_T$ from $u$ to $v$. It follows that $T \cap A$ lies in one component of $\Gamma_T$. Similarly, $T \cap B$ lies in one component of $\Gamma_T$. Moreover, by construction of $\Lambda'_A$ there certainly exists an edge of $\Gamma$ from a vertex in $T \cap A$  to a vertex in  $T \cap B$. Thus  $\Gamma_T$ is connected, as desired.

Finally, if $n$ is odd, we apply the above argument to the subgraph $\Gamma'$ of the random graph $\Gamma$ induced by $V - \{n\}$. It implies that $\Gamma'$ almost surely contains a lawful perfect antimatching $\mcl{M}$. Then $\mcl{M} \cup \{\{n\}\}$ is almost surely a legal system on $\Gamma$, as $n$ is almost surely joined to both vertices of some $S\in\mcl M$.
\end{proof}	

We are grateful to  Roman Glebov and Gonzalo Fiz Pontiveros for describing arguments that explain how to handle random graphs at high density.
In particular, the following is due to Roman Glebov:
\begin{lem}\label{lem:high}
Let $q=1-p$.
Suppose $$q = O(\frac{\log n}{n}).$$
Then a graph $\Gamma$ in $\mathcal G(n,p)$ almost surely either is a complete graph or contains a legal system.
\end{lem}
\begin{proof}
 Let $H$ be the complement of $\Gamma$. Note that $H$ is in $\mathcal G(n,q)$. We consider two cases.
\item
\textbf{Case 1: $q = o(n^{-3/2})$.}\\
The probability that a triple of vertices is joined by two edges in $H$ equals $3q^2-2q^3 \leq  3q^2$.
The expected number of such triples is of order ${n\choose 3}3q^2 \leq \frac12 n^3q^2\to 0$ as $n\to\infty$. Thus the probability that there is a pair of edges in $H$ with a common vertex tends to $0$ as $n\to\infty$.
We now describe a winning system when $\Gamma\in \mathcal G(n,p)$ is not a complete graph. Choose an edge in $H$ joining vertices $v,w$. Consider the system of moves consisting of $\{v,w\}$ and the singletons of all other vertices. Let the initial state $S = \{v\}$. Since every other vertex in $\Gamma$ is joined by an edge to both $v$ and $w$, the orbit of $S$ is legal.
\item
\textbf{Case 2: $q\neq o(n^{-3/2})$.}\\
Let $k$ be the number of edges in a maximal matching $M$ in $H$.
Denote by $D$ the set of all vertices that are not incident to any edge in $M$. Note that $|D| = n-2k$ and the subgraph $\Gamma(D)$ of $\Gamma$ is a clique. Consider the system of moves consisting of $\{v,w\}$ for all $(v,w)\in M$ and $\{v\}$ for all $v\in D$. Let $S$ be a state that has exactly one vertex in each $\{v,w\}$ for $(v,w)\in M$.

We first prove that $\Gamma(S)$ is connected. Suppose to the contrary that $\Gamma(S)$ can be decomposed
 into two subgraphs $\Gamma(S_1), \Gamma(S_2)$ not having adjacent vertices,
 Then the subgraph $H(S)$ induced by $S$ in $H$ contains the complete bipartite graph on $S_1,S_2$  as a subgraph.
 Thus there is a vertex in $H(S)$ of valence at least $\lceil\frac k 2\rceil$.

Let $d_H$ denote the maximum valence of $H$, and note that $d_H=O(qn)$ by Theorem~\ref{thm:maxdegree}. By hypothesis $q = O(\frac{\log n}{n})$, and so $d_H=O(\log n)$.
Combining with the previous conclusion about a vertex of $H(S)$, we find that $k=O(\log n)$.

  The set $D$ of vertices not incident to any edge in a largest matching has $n-2k$ vertices and the subgraph $H(D)$ has no edges. We will use that $k=O(\log n)$ to show below that the edgeless subgraph $H(D)$ almost surely does not exist.

  The probability that a set of $n-2k$ vertices has no edges equals $(1-q)^{\binom{n-2k}{2}}$,
   and hence the probability that there exists a subset of size $n-2k$ with no edges is bounded by $$\binom{n}{2k}(1-q)^{\binom{n-2k}{2}} \leq n^{2k}e^{-q{\binom{n-2k}{2}}}.$$
As $k=O(\log n)$, for all sufficiently large $n$ we have:
 $$n^{2k}e^{-q{\binom{n-2k}{2}}}    \leq e^{2c\log^2n - q\binom{n-2c\log n }{2}} $$
And hence since $q\geq b n^{-3/2}$ with $b>0$ for sufficiently large $n$ we have:
$$2c\log^2n - q\binom{n-2c\log n}{2}\leq
2c\log^2n - b n^{-3/2}\binom{n-2c\log n}{2}\to -\infty$$
And so the above probability goes to 0 as claimed.

Having proven that $\Gamma(S)$ is connected whenever $S$ is a state containing exactly one vertex from each edge in $M$, it remains to prove that $\Gamma(S\cup D')$ is connected for any $D'\subset D$. The probability that there is $d\in D$ that is not adjacent to any vertex in $S$ is $(n-2k)q^k
\leq (n-2k)\left(c\frac{\log n}{n}\right)^k$.
This last term tends to $0$ as $n\to \infty$ since $k\to \infty$.
\end{proof}

In response to Theorem~\ref{thm:random}, Fiz Pontiveros, Glebov and Karpas have proven the following result in \cite{FPGK},
which definitively explains when a random graph has a legal system:
\begin{thm}[Fiz Pontiveros-Glebov-Karpas]\label{PGK} Suppose that
\[\frac{\log n+\log\log n +\omega(1)}{n}\leq p \leq 1-\frac{\omega(1)}{n^2}.\]
Then a random graph in $\mcl{G}(n,p)$ almost surely has a legal system.
\end{thm}
As they explained a random graph $\Gamma\in\mcl{G}(n,p)$ that is not complete almost surely has a legal system precisely when $\Gamma$ almost surely has the property that each of its vertices has valence at least $2$.

For a slightly different setting of random regular graphs it would be interesting to resolve the following question.

\begin{prob}
	Does there a exist a constant $d \in \mathbb{N}$ such that a random $d$-regular graph on $n$ vertices almost surely has a legal system?
\end{prob}

We now describe a simpler bipartite version of Theorem~\ref{thm:random}.
\begin{prop}
Suppose
$$p \geq \frac{2\log n + \omega(n)}{n}$$ for some $\omega(n)$ with $\omega(n) \to \infty$.
Then a graph in $\mcl{G}(n,n,p)$ almost surely contains a legal system.
\end{prop}
\begin{proof}
First suppose $n$ is even. Let $V = A\sqcup B$ be the bipartite structure of a graph $\Gamma$ in $\mcl G(n,n,p)$. Partition $A$ into $A_1\sqcup A_2$ where $|A_i|= \frac{n}2$ for $i= 1,2$. Similarly partition $B$. Each of the graphs induced by $A_i \cup B_j$ for $i,j \in \{1,2\}$ is a random graph in $\mcl G(\frac n 2, \frac n 2, p)$. By Theorem~\ref{thm:Gnp1} they are all connected. Thus the system of moves corresponding to the bipartition and its orbit consisting of states $A_i\cup B_j$ is a legal system.

If $n$ is odd, we use the same argument for the subgraph $\Gamma'$ of $\Gamma$ induced by $n-1$ vertices in each side of the bipartition. The remaining vertices $a,b$ are almost surely joined to a vertex in $B_1$ and a vertex in $B_2$, or a vertex in $A_1$ and a vertex in $A_2$ respectively. Thus the orbit of the state $A_1\cup B_1\cup \{a,b\}$ is legal.
\end{proof}

We close with a few further questions about the existence of legal systems for certain families of graphs:

\begin{prob}Does a generic 3-connected finite planar graph $\Gamma$ with $\girth(\Gamma)=4$ and $\curvature(\Gamma)=0$
have the property that there is no legal state?
\end{prob}

\begin{prob}
Is there a sensible statement (positive or negative) that can be made for generic bipartite graphs (or girth~$4$ graphs) with a fixed number of vertices and  $\curvature(\Gamma)=1-v/2+e/4 =0$ ?

How about if $\Gamma$ is planar?
\end{prob}

\bibliographystyle{alpha}
\bibliography{C:/Users/Dani/Dropbox/papers/wise}

%
%
\end{document}